\documentclass[11pt]{article}
\usepackage[latin1]{inputenc}
\usepackage{amsmath}
\usepackage{amsfonts}
\usepackage{amssymb}
\usepackage{graphicx}
\usepackage{amsmath}
\usepackage{amsthm}
\usepackage{enumerate}
\usepackage{verbatim}
\usepackage{hyperref}
\usepackage{bbm}
\usepackage{color}

\numberwithin{equation}{section}
\def \N{\mathbb{N}}
\def \R{\mathbb{R}}
\def \Z{\mathbb{Z}}
\def \E{\mathbb{E}}
\def \P{\mathbb{P}}

\def \avirg{``}
\def\II{\mathcal{L}_{\mbox{{\tiny int}}} }
\def\I{\mbox{\large \bf 1}}
\def \D{\mathcal{D}}
\def \L{\mathcal{L}}

\def \dx{\Delta x}
\def\ii{\mathbf{i}}
\def\Id{\mathrm{Id}}
\def\tr{\mathrm{Tr}}
\def \pol{{\mathbf{\scriptstyle pol}}}
\theoremstyle{plain}
\newtheorem{theorem}{Theorem}[section]

\newtheorem{corollary}[theorem]{Corollary}

\newtheorem{lemma}[theorem]{Lemma}

\newtheorem{proposition}[theorem]{Proposition}
\newtheorem{remark}[theorem]{Remark}

\begin{document}
	\title{\bf Convergence rate of Markov chains and hybrid numerical schemes to jump-diffusions with application to the Bates model}
	
	\author{{\sc Maya Briani}\thanks{%
			Istituto per le Applicazioni del Calcolo, CNR Roma - {\tt m.briani@iac.cnr.it}}\\
		{\sc Lucia Caramellino}\thanks{%
			Dipartimento di Matematica,
			Universit\`a di Roma Tor Vergata, and INDAM-GNAMPA - {\tt caramell@mat.uniroma2.it}}\\
		{\sc Giulia Terenzi}\thanks{%
			Dipartimento di Matematica,
			Universit\`a di Roma Tor Vergata and Universit\'e Paris-Est, Laboratoire d'Analyse et de Math\'ematiques Appliqu\'es (UMR 8050), UPEM, UPEC, CNRS,  Projet Mathrisk INRIA, F-77454, Marne-la-Vall\'ee, France - {\tt terenzi@mat.uniroma2.it}}
	}
	\date{}
	\maketitle
	\begin{abstract}\noindent{\parindent0pt}
		We study the rate of weak convergence of Markov chains to diffusion processes under suitable but quite general assumptions. We give an example in the financial framework, applying the convergence analysis to  a multiple jumps tree approximation of the CIR process. Then, we  combine the Markov chain approach with other numerical techniques in order to handle the different components in jump-diffusion coupled models. We study the speed of convergence of this hybrid approach and we provide  an example in finance, applying our results to a tree-finite difference approximation in  the Heston or Bates model.
	\end{abstract}


	\smallskip
	
	
	\tableofcontents

	\section{Introduction}
	The paper is devoted to the study of the weak convergence rate of numerical schemes allowing one to handle specific jump-diffusion processes. These include the well known stochastic volatility models by Heston \cite{heston} and by Bates \cite{bates}. Since these dynamics involve the square root process for the volatility, a special numerical treatment has to be considered.   When dealing with European options, i.e. solutions to Partial (Integral) Differential Equation (hereafter P(I)DE) problems, numerical approaches involve tree methods \cite{ads, nv}, Monte Carlo procedures \cite{A-MC,A,AN, and, z}, finite-difference numerical schemes  \cite{ckmz,it, t} or quantization algorithms  \cite{PP}. 
	When American options are considered, that is, solutions to  specific optimal stopping problems or  P(I)DEs with obstacle, it is very useful to consider numerical methods which are able to easily handle dynamic programming principles, for example trees or finite-difference. We consider a numerical procedure which combines a tree method for the volatility process with a different  numerical approach for the asset price process, for instance finite-difference. Such a hybrid method has been developed and numerically studied in \cite{bcz,bcz-hhw,bctz} for the computation of European and American options in the stochastic volatility context. In this paper we study the rate of convergence. As a result, we can consider the Heston or the Bates model in the full parameter regime, differently from many other approaches.  Let us mention that, under these models, the literature is rich in numerical methods but, as far as we know, poor in results on the rate of convergence, with the exception of  the papers \cite{A,AN,bossy,z}, all them either dealing with schemes written on Brownian increments or requiring restrictions on the Heston diffusion parameters. So, we first study the convergence rate of tree methods and then we tackle the hybrid procedure.
	
	Tree methods rely heavily on Markov chains. So, in the first part (Section \ref{sect-markovapprox}) we study the rate at which a sequence of Markov chains weakly converges to a diffusion process $(Y_t)_{t\in[0,T]}$ solution to
	$$
	dY_t=\mu_Y(Y_t)dt+\sigma_Y(Y_t)dB_t.
	$$
	In this framework, the weak convergence  is well known to be governed by the behaviour of the local moments up to order 3 or 4 (see e.g. 
	\cite{SV}). In order to get the speed of convergence, we need to stress such requests, making further but quite general assumptions on the behaviour of the moments, and in Theorem \ref{conv_T} we  prove a  first order weak convergence result. As an application, we give an  example from the financial framework: we theoretically study the convergence rate of the  tree approximation proposed in \cite{acz} for the  CIR process. Recall that the  CIR process  \cite{cir} is a square root process, that is,
	$$
	dY_t=\kappa(\theta-Y_t)dt+\sigma\sqrt{Y_t}dB_t,
	$$
	with $\kappa,\theta,\sigma>0$. Recall also that this process lives in $[0,+\infty)$ and under the Feller condition $2\kappa\theta\geq\sigma^2$ it never hits 0. Several trees are considered in the literature, see e.g. \cite{cgmz, HST,Tian}, but  generally  some numerical problems arise when the Feller condition fails. Our result  for the tree in \cite{acz} (Theorem \ref{conv_T-CIR}) works in any parameter regime.
	 Recall that in equity markets, one often requires large values for the
vol-vol $\sigma$  whereas in interest rates context, $\sigma$ is markedly lower (see e.g. the calibration results in \cite{dps} and in \cite{bm} p. 115,  respectively).  So, a result in the full parameter regime is actually essential.
	
Let us mention that our general convergence Theorem \ref{conv_T}  may in principle be applied to more general trees constructed through the multiple jumps approach by Nelson and  Ramaswamy \cite{nr}, on which the tree in \cite{acz} is based  -- to our knowledge, a  theoretical study of the rate of convergence for such trees is missing in the literature. And it could also be used in other cases, e.g. the recent tree method developed in \cite{ads}.

	In the second part (Section \ref{sect-hybrid}), we link to $(Y_t)_{t\in[0,T]}$ a jump-diffusion process $(X_t)_{t\in[0,T]}$ which evolves according to a stochastic differential  whose coefficients only depend on the process $(Y_t)_{t\in[0,T]}$:
	$$
	dX_t=\mu_X(Y_t)dt+\sigma_X(Y_t)dW_t+\gamma_X(Y_t)dH_t,
	$$
	where $H$ is a compound Poisson process independent of the 2-dimensional Brownian motion $(W,B)$.
	So, the pair $(X_t,Y_t)_{t\in[0,T]}$ evolves following a Stochastic Differential Equation (hereafter SDE) with jumps. Given a function $f$, we consider the numerical computation of  $\E[f(X_T,Y_T)]$ through a generalization of the hybrid method introduced in \cite{bcz, bcz-hhw, bctz} (Section  \ref{sect-hybrid2}), which works backwardly by
	approximating the process $Y$ with a Markov chain and by using a different numerical scheme for solving a (local) PIDE allowing us to work  in the direction of the process $X$.
Then (Section \ref{sect-convergence}), in Theorem \ref{convergencebates} we give a general result on the rate of convergence of the hybrid approach.   We stress that the approximating algorithm is not directly written on a Markov approximation, so one cannot extend the convergence result provided in the first part of the paper. 
We then study  the stability and the consistency of the hybrid method,  but  in a sense that allows us to exploit the probabilistic properties of the Markov chain approximating the  process $Y$. 
	
	It is worth to be said that the test functions on which we study the rate of convergence are smooth. 
	In fact, there is a strict connection between such hybrid schemes and the use of a discrete noise in the approximation procedure. This means that we cannot use regularizing arguments \textit{\`a la Malliavin} in order to relax the smoothness requests, as it can be done when the approximation algorithm is based on the Brownian noise (see the seminal paper \cite{bt} or the recent \cite{AN} for the Heston model) or on a noise having at least a ``good piece of absolutely continuous part''  (Doeblin's condition, see \cite{br}).
	
	We then consider two possible finite-difference schemes (Section \ref{sect-finitedifference}) to handle the (local)  PIDE related to the component $X$:
	an implicit in time/centered in space 
	scheme (Section \ref{sect-l2}) and an implicit in time/upwind in space 
	scheme (Section \ref{sect-linf}). In both cases, the numerical treatment of the nonlocal term coming from the jumps involves implicit-explicit techniques, as well as numerical quadratures.  We apply the convergence Theorem \ref{convergencebates} and we obtain that  the hybrid algorithm has a rate of convergence of the first order in time and of a order in space according to the chosen numerical scheme.
	As an application, we give the weak convergence rate of the hybrid procedure written on the Heston and on the  Bates model (Section \ref{sect-bates}).
%

\vskip 0.5cm
\noindent{\small \textbf{Acknowledgment.}
LC acknowledges the MIUR Excellence Department Project awarded
to the Department of Mathematics, University of Rome Tor Vergata,
CUP E83C18000100006.}

	\section{Notation}\label{sect-notation}
In this section we establish the notation which will be used later on. Let $d\in \N^*=\N\setminus \{0\}$. 

\smallskip

$\bullet$ For a multi-index $l=(l_1,\dots,l_d)\in \N^d$ we define $|l|=\sum_{j=1}^{d}l_j$ and for $y\in \R^d$,
we define $\partial^l_y=\partial_{y_1}^{l_1}\cdots \partial_{y_d}^{l_d}$ and $y^l=y_1^{l_1}\cdots y_d^{l_d}$.
Moreover,   we denote by $|y|$ the standard Euclidean norm in  $ \R^d$ and for any linear operator $A: \R^d\rightarrow \R^d$, we denote by $|A|=\sup_{|y|=1}|Ay|$ the induced norm. 

\smallskip

$\bullet$ 
$L^p(\R^d,d\mathfrak m)$ denotes  the standard $L^p$-space w.r.t. the measure $\mathfrak m$ on $(\R^d,\mathcal{B}_d)$,  $\mathcal{B}_d$ denoting the Borel $\sigma$-algebra on $\R^d$, and we set  $|\cdot|_{L^p(\R^d, d\mathfrak m)}$  the associated norm. The Lebesgue measure is denoted through $dx$.

\smallskip

$\bullet$ Let  $\mathcal{D}\subseteq \R^d$ be a domain (possibly closed) and $q\in\N$.  $C^q(\mathcal{D})$ is the set of all functions on $\mathcal{D}$ which are $q$-times continuously differentiable.  We set 
$C^q_\pol(\mathcal{D} )$ 
the set of functions $g\in C^q(\mathcal{D})$ such that there exist $C,a>0$ for which 
\begin{equation*}
|\partial^{l}_yg(y)|\leq C(1+|y|^a),\qquad  y\in \mathcal{D}, \, |l|\leq q.
\end{equation*}
We set $C^{ q}_{\pol, T}(\mathcal{D})$ the set of functions $v\in C^{\lfloor q/2 \rfloor, q}([0,T)\times \mathcal{D} )$ such that there exist $C,a>0$ for which 
$$
\sup_{t<T}|\partial^{k}_t\partial^{l}_yv(t,y)|\leq C(1+|y|^a),\qquad y\in \mathcal{D} , \, 2k+|l|\leq q.
$$
For brevity, we set  $C(\mathcal{D})=C^0(\mathcal{D})$, 
$C_\pol(\mathcal{D} )=C^0_\pol(\mathcal{D} )$ and $C_{\pol, T}(\mathcal{D})=C^{0 }_{\pol, T}(\mathcal{D} )$.
We also need another functional space, that we call $C^{p,q}_{\pol}(\R^m,\mathcal{D})$, $p\in[1,\infty]$, $q\in\N$, $m\in\N^*$: $g=g(x,y)\in C^{p,q}_{\pol}(\R^m,\mathcal{D})$ if $g\in C_{\pol}^q(\R^m\times \mathcal{D})$ and  there exist $C,a>0$ such that
$$
|\partial^{l'}_x\partial^{l}_yg(\cdot, y)|_{L^p(\R^m,dx)}\leq C(1+|y|^a), \quad  |l'|+|l|\leq q.
$$
Similarly as above, we set $C^{p,q}_{\pol,T}(\R^m,\mathcal{D})$ the set of the function $v\in C^q_{\pol,T}(\R^m\times \mathcal{D} )$ such that 
$$
\sup_{t<T}|\partial^{k}_t\partial^{l'}_x\partial^{l}_yv(t,\cdot, y)|_{L^p(\R^m,dx)}\leq C(1+|y|^a), \quad  2k+|l'|+|l|\leq q.
$$

\smallskip

$\bullet$ 
For fixed $X_0=(X_{01},\ldots,X_{0d})\in\R^d$ and  $\dx=(\dx_1,\dots,\dx_d)\in(0,+\infty)^d$ (spatial step),  $\mathcal{X}=\{x=(X_{01}+i_1\Delta x_1,\dots, X_{0d}+i_d\Delta x_d)\}_{i\in \Z^d}$ denotes a discrete grid in $\R^d$. For $p\in[1,\infty]$, we set $l_p(\mathcal{X})$ as the discrete $l_p$-space of the functions $\varphi\,:\,\mathcal{X}\to\R$ with the norm $|\varphi|_{p}=(\sum_{x\in\mathcal{X}} |\varphi(x)|^{p}\dx_1\cdots\dx_d)^{1/p}$ if $p\in [1,\infty)$ and $|\varphi|_{\infty}=\sup_{x\in\mathcal{X}}|\varphi(x)|$ if $p=\infty$.  Moreover, for a linear operator  $\Gamma \,:\, l_p(\mathcal{X})\to l_p(\mathcal{X})$, the induced norm is denoted by $|\Gamma|_p=\sup_{|\varphi|_p\leq 1}|\Gamma \varphi|_{p}$.  
And for a function $g\,:\,\R^d\to\R$, we set $|g|_p$ the $l_p(\mathcal{X})$ norm of the restriction of $g$ on $\mathcal{X}$. When $d=1$, we identify $(\varphi(x))_{x\in\mathcal{X}}$ with $(\varphi_i)_{i\in\Z}$ through $\varphi_i=\varphi(X_0+i\Delta x)$, $i\in\Z$.

\smallskip

$\bullet$ 
$L^p(\Omega)$ is the short notation for the standard $L^p$-space on the probability space $(\Omega,\mathcal{F},\P)$, on which the expectation is denoted by $\E$. We set $\|\cdot\|_p$  the norm in $L^p(\Omega)$.

\section{First order weak convergence of Markov chains to diffusions}\label{sect-markovapprox}
Let $d\in \N^*$ and $\mathcal{D}\subseteq\R^d$ be a convex  domain or a closure of it. 
On a probability space $(\Omega, \mathcal F,\P)$, we consider a $d$-dimensional diffusion process driven by
\begin{equation}\label{Y}
dY_t=\mu_Y(Y_t)dt+\sigma_Y(Y_t)dW_t, \qquad Y_0 \in \mathcal{D},
\end{equation}
where $W$ is a $\ell$-dimensional standard Brownian motion. From now on, we set $a_Y=\sigma_Y\sigma_Y^\star$, the notation $\star$ denoting transpose. We recall that the associated infinitesimal generator is given by
\begin{equation}
\label{A-op}
\mathcal{A}=\frac 12 \tr (a_Y D_ y^2)+\mu_Y\cdot \nabla_y,
\end{equation}
where $\tr$ denotes the matrix trace, $D^2_y$ and $\nabla_y$ are, respectively, the Hessian and the gradient operator w.r.t. the space variable $y$ and the notation  ``$\cdot$'' stands for the scalar product.

Hereafter, we fix $T>0$, $f:\mathcal{D}\rightarrow \R$ and we define
\begin{equation}\label{u}
u(t,y)=\E[f(Y^{t,y}_T)],\quad (t,y)\in[0,T]\times\mathcal{D},
\end{equation}
where $Y^{t,y}$  denotes the solution to the SDE in \eqref{Y} that starts at $t$ in the position $y$. We do not enter in specific requests for the diffusion coefficients or for $f$, we just ask that the following properties are met:
\begin{itemize}
	\item [(a)]
	$\mu_Y$ has polynomial growth; 
	\item [(b)]	
	for every $(t,y)\in[0,T]\times \mathcal{D}$ there exists a unique weak solution $(Y^{t,y}_s)_{s\in[t,T]}$ of \eqref{Y} such that
	$\P(\forall s\in[t,T] , \,Y^{t,y}_s\in\mathcal{D})=1;$
	\item[(c)]
	the function  $u$ in \eqref{u} solves the PDE
\begin{equation}\label{PDE-u-Y}
\begin{cases}
\frac{\partial u }{\partial t}+\mathcal{A} u=0,\qquad& \mbox{in } [0,T)\times \D,\\
u(T,y)=f(y),  &\mbox{in }  \D.
\end{cases}	
\end{equation}	
\end{itemize}
The above proverties (a), (b) and (c) will be assumed to hold throughout this section.
 
We are interested in the numerical evaluation of $u(0,Y_0)=\E(f(Y_T))$.
A widely used and computationally convenient method  is by computing the above expectation on an approximation of the process $Y$. Here, we consider an approximation through a  Markov chain that weakly converges to the diffusion process $Y$, see e.g. the classical references 
\cite{SV}.  We will see in Section \ref{sect-CIR} an application to tree methods, that is, when the process $Y$ is approximated by means of a computationally simple Markov chain. 
Here, our aim is to study, under suitable but quite general assumptions, the order of weak convergence.

%

So,  let $N\in\N^*$ and set $h=T/ N$. The parameters $N$ and $h$ are fixed once for all. Let $( Y^h_{n})_{n=0,\ldots,N}$ denote a Markov chain, 
whose state space, at time-step $n$, is given by $\mathcal{Y}_n^h\subset \mathcal{D}$. In our mind, $( Y^h_{n})_{n=0,\ldots,N}$ is a Markov process which is a discrete weak  approximation in time (and possibly in space) of the $d$-dimensional diffusion $Y$, namely, $Y^h_n$ approximates $Y$ at times $nh$, for every $n=0,\ldots,N$. Of course, we assume that $Y^h_0=Y_0$, that is, $\mathcal{Y}^h_0=\{Y_0\}$.
Without loss of generality, we may assume that $( Y^h_{n})_{n=0,\ldots,N}$ is defined in $(\Omega,\mathcal{F},\P)$.

In order to study the rate of the weak convergence of $(Y^h_n)_{n=0,\dots,N}$ to $Y$, we need to stress the requests that   are usually done in order to merely prove the convergence (see e.g. 
\cite{SV}). In particular, we need the following assumption.
\medskip

\noindent
\textbf{Assumption $\mathcal{H}_1$.}
\textit{There exists $	\bar h >0$ such that, for every $h< \bar h$,  the first three local moments satisfy
	\begin{align}
	\E[Y^h_{n+1}-Y^h_n\mid Y^h_n]
	&=\mu_Y(Y^h_n)h+f_h(Y^h_n), \label{momento1-gen}\\
	\E[( Y^h_{n+1}-Y^h_n)( Y^h_{n+1}-Y^h_n)^\star\mid Y^{h}_n]
	&=a_Y(Y^{h}_n) h 
	+g_h(Y^{h}_n),\label{momento2-gen}\\
	\E[( Y^h_{n+1}-Y^h_n)^l\mid Y^{h}_n]
	&=j_{h,l}(Y^{h}_n), \qquad l\in\N^d, \,|l|=3\label{momento3-gen},
	\end{align}
	where  $
	f_h:\mathcal{D}\rightarrow \R^d$, $g_h:\mathcal{D}\rightarrow \R^{d\times d} $ and $j_{h,l}:\mathcal{D}\rightarrow \R$ satisfy the following properties:  there exist $p>1$ and $C>0$ such that
	\begin{align}
	&\quad \sup_{h\leq \bar h }	\sup_{n=0,\dots, N}	\|f_h(Y^h_n)\|_p\leq Ch^2
	,\label{stima-f}\\
	&\quad \sup_{h\leq \bar h }	\sup_{n=0,\dots, N}\|g_h(Y^h_n)\|_p\leq Ch^2
	,\label{stima-g}\\
	&\quad \sup_{h\leq \bar h }	\sup_{n=0,\dots, N}\|j_{h,l}(Y^h_n)\|_p\leq Ch^2, \quad |l|=3.\label{stima-j}
	\end{align}
}

We also need  the following behavior of the moments.

\medskip

\noindent
\textbf{Assumption $\mathcal{H}_2$.}
\textit{There exists $	\bar h >0$ such that for every $p>1$ there exists $C_p>0$ for which
	\begin{align}
	&\label{stima_momento-gen}
	\sup_{h<\bar h}\sup_{0\leq n \leq N} \|Y^h_n\|_p\leq C_{p},\\
	&\label{stima_incremento-gen}
	\sup_{h<\bar h}\sup_{0\leq n\leq N} \frac{1}{\sqrt h} \|Y^h_{n+1}-Y^h_n\|_p\leq C_p.
	\end{align}
}

\medskip

We can now state the following first order weak convergence result.

\begin{theorem} \label{conv_T}
	Let assumptions $\mathcal{H}_1$ and $\mathcal{H}_2$ hold and assume that  $u\in C^{4}_{\pol,T}(\D)$, $u$ being defined in \eqref{u}.
	Then  there exist $\bar h>0$ and  $C>0$ such that for every $h<\bar h$ one has
	$$
	|	\E[f(Y^h_N)]-\E[f(Y_T)]|\leq CTh.
	$$
\end{theorem}
\begin{proof}
	The proof is quite standard.	Since  $\E[f(Y^h_N)]=\E[u(T,Y^h_T)]$ and $  \E[f(Y_T)] =u(0,Y_0)$, we have
	$$
	\E[f(Y^h_T)]- \E[f(Y_T)] =\E[  u(T,Y^h_T)-u(0,Y_0) ]=\sum_{n=0}^{N-1}\E[  u((n+1)h,Y^h_{n+1})-u(nh,Y^h_n) ].
	$$
	Since $u\in C^{4}_{\pol,T}(\D)$, we can  apply Taylor's formula to $t\mapsto u(t,y)$ around $nh$    up to order 1 and to the functions $y\mapsto u(t,y)$ and $y\mapsto \partial_t u(t,y)$ around $Y^h_n$ up to order 3 and 1 respectively. We  obtain
	\begin{equation}\label{app1}
	u((n+1)h,Y^h_{n+1})
	= \sum_{ 
		0\leq |l|+2l'\leq 3} \partial_y^{l}\partial_t^{l'} u(nh,Y^h_n)\frac{h^{l'}(Y^h_{n+1}-Y^h_n)^{l}}{|l|!l'!} +R_1(n,h,Y^h_n,Y^h_{n+1}),
	\end{equation}
	where the remaining term $R_1$ is given by
$$
	\begin{array}{ll}
	R_1(n,h,Y^h_n,Y^h_{n+1})	
	&=h^2\int_0^1 (1-\tau)\partial^2_tu(t+\tau h,Y_{n+1}^h)d\tau\\
	&+h\sum_{|k|=2}(Y^h_{n+1}-Y^h_n)^k\int_0^1 (1-\xi)\partial^k_y\partial_tu(nh,Y^h_n+\xi(Y^h_{n+1}-Y^h_n))d\xi\\
	&+\sum_{|k|=4}\frac{(Y^h_{n+1}-Y^h_n)^k}{3!}\int_0^1 (1-\xi)^3\partial^k_yu(nh,Y^h_n+\xi(Y^h_{n+1}-Y^h_n))d\xi.
	\end{array}
$$
	We now pass to the conditional expectation w.r.t. $Y^h_n$ in \eqref{app1} and use \eqref{momento1-gen} and \eqref{momento2-gen}. By rearranging the terms we obtain
	\begin{equation}\label{resti}
	\begin{array}{ll}
	\E[u((n+1)h,Y^h_{n+1})-u(nh,Y^h_n)] &=h\E\left[ \partial_tu(nh,Y^h_n)+\mu_Y(Y_n^h)\cdot \nabla_y u(nh,Y^h_n)
	+\frac 12 \mbox{Tr}(a_YD^2_y  u(nh,Y^h_n))\right]\\
	&\quad +R^{1}_n(h)+R^{2}_n(h)+R^{3}_n(h)+R^{4}_n(h)+R^{5}_n(h),
	\end{array}	
	\end{equation}
	in which
$$
\begin{array}{ll}
\displaystyle
R^1_{n}(h)=\E[ R_1(n,h,V^h_n,V^h_{n+1})],
&\quad
R^{2}_n(h)=h\E[(\mu_Y(Y^h_n)h+ f_h(Y^h_n))\cdot \nabla_y\partial_tu(nh,Y^h_{n})],\\
R^{3}_n(h)=\E[f_h(Y^h_n)\cdot \nabla_yu(nh,Y^h_{n})],
&\quad
R^{4}_n(h)=\frac 12\E[\tr(g_h(Y^h_n)D^{2}_yu(nh,Y^{h}_n))],\\
R^{5}_n(h)=\frac 16\sum_{|k|=3}\E[\partial_y^ku(nh,Y^h_{n})j_{h,k}(Y^h_n)].&
\end{array}
$$

	Thanks to \eqref{PDE-u-Y},  the first term in \eqref{resti} is null, so
	$$
	|\E[u((n+1)h,Y^h_{n+1})-u(nh,Y^h_n)]|
	\leq \sum_{i=1}^{5}|R^{i}_n(h)|.
	$$
	We now prove that $|R^{i}_n(h)|\leq C h^2$, for every $i=1,\ldots 5$. Let $\bar h>0$ such that both assumptions $\mathcal{H}_1$ and $\mathcal{H}_2$ hold and let $h<\bar h$.
	Since the derivatives of $u$ have polynomial growth,  one has
	\begin{align*}
	&|R_1(n,h,Y^h_n,Y^h_{n+1})|\leq C\Big(1+|Y^h_n|+|Y^h_{n+1}|\Big)^{a}\Big[h^2+h|Y^h_{n+1}-Y^h_n|^2+|Y^h_{n+1}-Y^h_n|^4\Big],
	\end{align*}
	where $C,a>0$ denote constants that are independent of $h$ and, from now on, may change from a line to another.
	Then, by using the Cauchy-Schwarz inequality, \eqref{stima_momento-gen} and \eqref{stima_incremento-gen}, we get
	\begin{align*}
	|R^{1}_n(h)|&\leq C \big\|(1+|Y^h_{n+1}|+|Y^h_n|)^{a}\big\|_2\, \big\|h^2+h (Y^h_{n+1}-Y^h_n)^2+(Y^h_{n+1}-Y^h_n)^4 \big\|_2\leq C h^2.
	\end{align*}
	As regards $R^2_n(h)$, we use the polynomial growth of $\nabla_y\partial_tu$, the Cauchy-Schwarz inequality and the H\"older inequality, so that
	\begin{align*}
	|R^{2}_n(h)|&
	\leq C \E[\big(1+|Y^h_{n}|^{a}\big)|\mu_Y(Y^h_n)|]\,h^{2}+ C\E[\big(1+|Y^h_{n}|^a\big)|f_h(Y^h_n)|]\,\\&
	\leq C\big\|1+|Y^h_{n}|^{a}\big\|_2
	\,\big\|\mu_Y(Y^h_n)\big\|_2\,h^{2}+ C\big\|1+|Y^h_{n}|^{a}\big\|_q
	\,\big\|f_h(Y_n^h)\big\|_p,
	\end{align*}
	where $p$ is given in \eqref{stima-f} and $q$ is its conjugate exponent.  Since $\mu_Y$ has polynomial growth,  by \eqref{stima-f} and \eqref{stima_momento-gen} we get
	$$
	|R^{2}_n(h)|\leq C h^2.
	$$
	The remaining terms $R^3_n(h)$, $R^4_n(h)$ and $R^5_n(h)$ can be handled similarly, so 
	the statement follows.
\end{proof}

\subsection{An example: a first order weak convergent binomial tree for the CIR process}\label{sect-CIR}
We now fix $d=1$ and $\mathcal{D}=\R_+=[0,\infty)$. We consider the well known CIR process $(Y_t)_{t\in [0,T]}$
solution to the SDE
\begin{align*}
&dY_t= \kappa(\theta- Y_t)dt+\sigma\sqrt{Y_t}\,dB_t,
\quad Y_0\geq0.
\end{align*}
We assume that $\theta,\kappa,\sigma>0$ and we stress that we never require  the Feller condition $2\kappa\theta\geq \sigma^2$, ensuring that the process $Y$ does not hit $0$. Therefore,  the process  $Y$ can reach $0$.

The CIR process is widely used in finance  to model interest rates or the volatility process in stochastic volatility models and there is a large literature on numerical methods to approximate it, see e.g. \cite{A-MC, cgmz, HST, W}.
%
We consider  here the  \avirg multiple jumps"  tree approximation for the  CIR process developed in \cite{acz}. 
%
We first recall how the tree works and then, as an application of Theorem \ref{conv_T}, we study the rate of convergence.

For $n=0,1,\ldots,N$ consider the lattice
\begin{equation}\label{vnk}
\mathcal{Y}_n^h=\{y^n_k\}_{k=0,1,\ldots,n}\quad\mbox{with}\quad
y^n_k=\Big(\sqrt {Y_0}+\frac{\sigma} 2(2k-n)\sqrt{h}\Big)^2\I_{\{\sqrt {Y_0}+\frac{\sigma} 2(2k-n)\sqrt{h}>0\}}.
\end{equation}
Note that $\mathcal{Y}_0^h=\{Y_0\}$. Moreover, the lattice is binomial recombining and, for $n$ large, the ``small'' points degenerate at $0$.
For each fixed node $(n,k)\in\{0,1,\ldots,N-1\}\times\{0,1,\ldots,n\}$, the ``up'' jump $k_u(n,k)$ and the ``down'' jump $k_d(n,k)$ from $y^n_k\in\mathcal{Y}_n^h$ are defined as 
\begin{align}
\label{ku}
&k_u(n,k) =
\min \{k^*\,:\, k+1\leq k^*\leq n+1\mbox{ and }y^n_k+\mu_Y(y^n_k)h \le y^{n+1}_{k^*}\},\\
\label{kd}
&k_d(n,k)=
\max \{k^*\,:\, 0\leq k^*\leq k \mbox{ and }y^n_k+\mu_Y(y^n_k)h \ge y^{n+1}_{ k^*}\},
\end{align}
where $\mu_Y(y)=\kappa(\theta-y)$ and
with the understanding $k_u(n,k)=n+1$, resp. $k_d(n,k)=0$,  if the set in \eqref{ku}, resp. \eqref{kd}, is empty. 
This is called the ``multiple jump approach'': the up jump can be larger than the closest up node, and similarly, the down jump can be smaller than the closest down node. This is as opposed to the ``single jump approach'', where typically $k_d(n,k)=k$ and $k_u(n,k)=k+1$. The multiple jumps have been smartly introduced in \cite{nr} and are very useful because they allow one to define the transition probabilities such that the local first moment is asymptotically best fit. In fact, starting from the node $(n,k)$ the probability that the process jumps to $k_u(n,k)$ and $k_d(n,k)$ at time-step $n+1$ are set as
\begin{equation}\label{pik}
p_u(n,k)
=0\vee \frac{\mu_Y(y^n_k)h+ y^n_k-y^{n+1}_{k_d(n,k)} }{y^{n+1}_{k_u(n,k)}-y^{n+1}_{k_d(n,k)}}\wedge 1
\quad\mbox{and}\quad p_d(n,k)=1-p_u(n,k)
\end{equation}
respectively. We will see in next Proposition \ref{propjumps} that for $h$ small enough the parts ``$0\vee$'' and ``$\wedge 1$'' can be omitted.

We call $(Y^h_n)_{n=0,1,\ldots,N}$ the Markov chain governed by the above jump probabilities. As an application of Theorem \ref{conv_T}, we shall prove the following result.

\begin{theorem}\label{conv_T-CIR}
	Let $f\in C^{4}_\pol(\R_+)$. Then, 
	there exist $\bar h>0$ and $C>0$ such that for every $h<\bar h$,
	$$
	|	\E[f(Y^h_N)]-\E[f(Y_T)]|\leq CTh,
	$$
	that is, the tree approximation $(Y^h_n)_{n=0,\dots,N}$ is first order weak convergent.
\end{theorem}

In order to discuss the  assumptions $\mathcal{H}_1$ and $\mathcal{H}_2$  of Theorem \ref{conv_T}, we need some preliminary results which pave the way to the analysis of the convergence. 

\begin{proposition}\label{propjumps}
	There exist  $\theta_*,\theta^*, C_*, \bar h>0$ such that for any $h<\bar h$ the following properties hold.
	\begin{enumerate}
		\item If $\theta_*h\leq y^n_k \leq  \theta^*/h$, then $	k_u(n,k)=k+1$, $k_d(n,k)=k$. Moreover,
		\begin{equation*}
		y^{n+1}_{k_u(n,k)}=y^n_k+\frac{\sigma^2}{4}h +\sigma \sqrt{y^n_k h }
		\quad\mbox{and}\quad
		y^{n+1}_{k_d(n,k)}=y^n_k+\frac{\sigma^2}{4}h -\sigma \sqrt{y^n_k h }.
		\end{equation*}
		
		\item If $y^n_k < \theta_*h$, then $k_d(n,k)=k$. Moreover,
		\begin{equation}\label{stimasottosoglia}
		0\leq y^{n+1}_{k_u(n,k)}-y^n_k\leq C_*h.
		\end{equation}
		
		\item  If $y^n_k > \theta^*/h$, then $k_u(n,k)=k+1$.
		
		\item The jump probabilities are
		\begin{equation}\label{proba}
		p_u(n,k)
		= \frac{\mu_Y(y^n_k)h+ y^n_k-y^{n+1}_{k_d(n,k)} }{y^{n+1}_{k_u(n,k)}-y^{n+1}_{k_d(n,k)}},
		\quad \quad p_d(n,k)=    \frac{ y^n_{k_u(n,k)}-y^{n}_{k}-\mu_Y(y^n_k)h }{y^{n+1}_{k_u(n,k)}-y^{n+1}_{k_d(n,k)}}.
		\end{equation}
		
	\end{enumerate}
\end{proposition}

The proof of Proposition \ref{propjumps} relies in a boring study of the properties of the lattice, so we postpone it in Appendix \ref{app-jumps}.
This is all we need to prove that $\mathcal{H}_2$ holds:

\begin{proposition}\label{moments-2} 
	The CIR approximating tree $\{Y^h_n\}_{n=0,\ldots,N}$ satisfies Assumption $\mathcal{H}_2$. 
\end{proposition}

\begin{proof}
\textbf{Step 1: proof of \eqref{stima_momento-gen}.}
We use a technique  firstly developed in \cite{A-MC} for a CIR discretization scheme based on Brownian increments. The key point is the proof of a monotonicity property allowing one to control the moments of the tree:  there exist $b,C, \bar h>0$ such that for every $h < \bar h$ and $n=0,\dots,N-1$ one has
	\begin{equation}\label{key}
	0\leq Y^h_{n+1}\leq (1+bh)Y^h_n+Ch+\sigma\sqrt{Y^h_nh} \, W^h_{n+1},
	\end{equation}
	where $W^h_{n+1}$ is a r.v. such that 
	\begin{equation}\label{law-W}
	\P(W^h_{n+1}=2p_d(n,k)|Y^h_n=y^n_k )=p_u(n,k)=1-\P(W^h_{n+1}=-2p_u(n,k)|Y^h_n=y^n_k ).
	\end{equation}
To this purpose,  fix a node $(n,k)$. For the sake of simplicity, we write $k_u$, resp. $k_d$, in place of $k_u(n,k)$, resp. $k_d(n,k)$. We have (see \eqref{tree1}) that
$$
	y^{n+1}_{k+1}\leq y^n_k+\frac{\sigma^2}{4}h +\sigma \sqrt{y^n_k h },\qquad
	y^{n+1}_{k}\leq y^n_k+\frac{\sigma^2}{4}h -\sigma \sqrt{y^n_k h }.
$$
	By Proposition \ref{propjumps}, for $h<\bar h$,
	if $\theta_*h< y^n_k <\theta^*/h$  the up and down jumps are  both single, hence $y^{n+1}_{k_u}=y^{n+1}_{k+1}$ and $y^{n+1}_{k_d}=y^{n+1}_{k}$
	On the other hand, if $y^n_k \geq \theta^*/h$ the up jump is single, that is $y^{n+1}_{k_u}=y^{n+1}_{k+1}$ , 
	while the down jump can be multiple but, in every case, is still true that
	\begin{equation*}
	y^{n+1}_{k_d}\leq	y^{n+1}_{k}=  y^n_k+\frac{\sigma^2}{4}h -\sigma \sqrt{y^n_k h }.
	\end{equation*}
	Finally, if $y^n_k \leq \theta_*h$, we have $y^{n+1}_{k_d}=y^{n+1}_{k}$, 
	while the up jump can be multiple but we can always write
	$$
	y^{n+1}_{k_u} \leq y^n_k+C_*h \leq y^n_k+C_*h+ \sigma \sqrt{y^n_k h}.
	$$
	Summing up, if we set $\bar C=\max\Big(C_*, \frac{\sigma^2}{4}\Big)$,
	for every $h$ small 
	we can write
	$$
	0\leq Y^h_{n+1}\leq Y^h_n+ \bar Ch+\sigma \sqrt{Y^h_n h }\, Z^h_{n+1},
	$$
	where $Z^h_{n+1}$ is a random variable such that $\P( Z^h_{n+1}=+1|Y^h_n=y^n_k )=p_u(n,k)$ and  $\P( Z^h_{n+1}=-1|Y^h_n=y^n_k )=p_d(n,k)$. 
	Note that $\E(Z^h_{n+1}|Y^h_n=y^n_k )=p_u(n,k)-p_d(n,k)=2p_u(n,k)-1$. Then,  the random variable
	$$
	W^h_{n+1}= Z^h_{n+1}-\E[Z^h_{n+1}|Y^h_n]
	$$
	has  exactly the law given in \eqref{law-W}.
	We also define  the function $	P_u(y^n_k)= p_u(n,k)$.
	Therefore,
$$
	\begin{array}{rl}
	 0\leq &	Y^h_{n+1}\leq 
	Y^h_n+ \bar Ch+\sigma \sqrt{Y^h_n h} \,(2P_u(Y^h_n)-1)  +\sigma \sqrt{Y^h_n h }\,W^h_{n+1}\\
	\leq &
	Y^h_n+ \bar Ch+\sigma \sqrt{\theta^*}\sqrt{\frac{Y^h_n h}{\theta^*} } \,\big|2P_u(Y^h_n)-1\big|\I_{ \{Y^h_n  \geq \frac{\theta^*}{h}\}}
	+\sigma \sqrt{Y^h_n h} \,\big(2P_u(Y^h_n)-1\big)\I_{ \{Y^h_n < \frac{\theta^*}{h}\}}\\
	&+\sigma \sqrt{Y^h_n h }\,W^h_{n+1}.
	\end{array}
$$
	Now, if $Y^h_n  \geq \frac{\theta^*}{h}$ then $\sqrt{\frac{Y^h_n h}{\theta^*} }
	\leq \frac{Y^h_n h}{\theta^*}$ and, since $P_u\in[0,1]$, we have $|2P_u(Y^h_n)-1|\leq 1$. Then, we have
	$$
	0\leq Y^h_{n+1}
	\leq
	(1+bh)Y^h_n+ \bar Ch
	+\sigma \sqrt{Y^h_n h} \,\big(2P_u(Y^h_n)-1\big)\I_{ \{Y^h_n < \frac{\theta^*}{h}\}}
	+\sigma \sqrt{Y^h_n h }\,W^h_{n+1},
$$
	where $b=\frac{\sigma}{\sqrt{\theta^*}}$.
	Let us study the quantity $\sigma \sqrt{Y^h_n h }\, (2P_u(Y^h_n)-1) \I_{ \{Y^h_n  < \frac{\theta^*}{h}\}}  $. If   $\theta_*h< y^n_k < \theta^*/h$, by using \eqref{proba} and point 1. of Proposition \ref{propjumps}, we can explicitly write
	\begin{align*}
	\sigma \sqrt{y^n_k h }\,&(	2P_u(y^n_k)-1)
	= 	\sigma \sqrt{y^n_k h }\,\Big(2\Big(  \frac 1 2 + \frac{4\mu_Y(v_k^n)-\sigma^2}{8\sigma \sqrt{y^n_k h }}   \Big)h-1
	\Big)=	\mu_Y(v_k^n)h-\frac{\sigma^2}4h \leq  \kappa\theta h.	
	\end{align*}
	If instead $y^n_k \leq \theta_*h$, then by using 2. in Proposition \ref{propjumps} we have
$$
	\begin{array}{rl}
	\sigma \sqrt{y^n_k h }\,(	2P_u(y^n_k)-1)
	&=	\sigma \sqrt{y^n_k h }\,  \frac{2\mu_Y(y^n_k)h+ 2y^n_k-y^{n+1}_{k_d(n,k)}-y^{n+1}_{k_u(n,k)} }{y^{n+1}_{k_u(n,k)}-y^{n+1}_{k_d(n,k)}} \\
	&\leq	\sigma \sqrt{y^n_k h }\, \frac{2\mu_Y(y^n_k)h+ 2y^n_k }{y^{n+1}_{k+1}-y^{n+1}_{k}} \leq	\sigma \sqrt{y^n_k h }\,  \frac{2\kappa\theta h+ 2\theta_*h }{2\sigma \sqrt{y^n_k h }} =(\kappa\theta + \theta_*)h.	
	\end{array}
$$
	So, by inserting, for every $n\leq N-1$ we get
$$
	\begin{array}{rl}
	\nonumber	0\leq Y^h_{n+1}&\leq (1+bh)Y^h_n+\bar Ch+\sigma (\kappa\theta + \theta_*) h
	+\sigma \sqrt{Y^h_n h }\,W^h_{n+1}
	\end{array}
$$
	and \eqref{key} is proved.
	
	Now, by using \eqref{key} and \eqref{law-W}, we can repeat step by step the proof of Lemma 2.6 in \cite{A-MC} and we get \eqref{stima_momento-gen}.
	
	\medskip 	 
	\textbf{Step 2: proof of \eqref{stima_incremento-gen}}. 	 We can write
$$
	\begin{array}{rl}
	| Y^h_{n+1}- Y^h_n|^p\leq &3^{p-1}\Big| \frac{\sigma^2}{4}h +\sigma\sqrt{ Y^h_nh}Z^h_{n+1}\Big|^p\I_{ \{ \theta_*h<  Y^h_n< \theta^*/h\}}
	+3^{p-1}|  Y^h_{n+1}- Y^h_n|^p\I_{ \{  Y^h_n\leq \theta_*h\}}\\&+3^{p-1}| Y^h_{n+1}- Y^h_n|^p\I_{\{ Y^h_n\geq \theta^*/h\}}=:3^{p-1}(I_1+I_2+I_3),
	\end{array}
$$
	where we have used that, on the set $\{\theta_*h<  Y^h_n< \theta^*/h\}$, we have
	$
	Y^h_{n+1}=  Y^h_n+ \frac{\sigma^2}{4}h +\sigma\sqrt{ Y^h_nh}Z^h_{n+1},
	$
	with $\P(Z^h_{n+1}=1\mid  Y^h_{n+1})=P_u( Y^h_n)$ and $\P(Z^h_{n+1}=-1\mid  Y^h_{n+1})=P_d( Y^h_n)$.
	Now, by using \eqref{stima_momento-gen}, Proposition \ref{propjumps}, the Cauchy-Swartz and the Markov inequality,
$$
	\begin{array}{rl}
	I_1&\leq \E\Big[\Big(\frac{\sigma^2}{4}h +\sigma\sqrt{ Y^h_nh}\Big)^p\Big]
	\leq
	2^{p-1}\Big(\Big(\frac{\sigma^2}{4}\Big)^p+\sigma^p\E[( Y^h_n)^p]^{1/2} \Big)h^{p/2}
	\leq
	2^{p-1}\Big(\Big(\frac{\sigma^2}{4}\Big)^p+\sigma^p\sqrt{C_p} \Big)h^{p/2},\\
	I_2 &\leq C_*^ph^p,\\
	I_3	&\leq \E[( Y^h_{n+1}- Y^h_n)^{2p}]^{1/2} \P\Big( Y^h_n>\frac{\theta^*}{h}\Big)^{1/2}\leq 2^{p}\sqrt{\frac{C_{2p}C_p}{(\theta^*)^p} }\,h^{p/2},
	\end{array}
$$
	and \eqref{stima_incremento-gen} follows.
\end{proof}

\begin{proposition}\label{moments-1} 
	The CIR approximating tree $\{Y^h_n\}_{n=0,\ldots,N}$ satisfies Assumption $\mathcal{H}_1$.
	
\end{proposition}

\begin{proof}
	Straightforward computations give $	\E[Y^h_{n+1}-Y^h_n\mid Y^h_n]=\mu_Y(Y^h_n)h$, so \eqref{momento1-gen} and \eqref{stima-f} immediately follow. As for	\eqref{momento2-gen}, 
	\begin{align*}
	&\E[( Y^h_{n+1}-Y^h_n)^2\mid Y^h_n=y^n_k]  =\E[( Y^h_{n+1}-Y^h_n)^2\mid Y^h_n=y^n_k]\I_{\{y^n_k \leq \theta_*h\}}\\
	&\quad +\E[( Y^h_{n+1}-Y^h_n)^2\mid Y^h_n=y^n_k]\I_{\{\theta_*h\leq y^n_k \leq \theta^*/h\}} 
	+\E[( Y^h_{n+1}-Y^h_n)^2\mid Y^h_n=y^n_k] \I_{\{y^n_k > \theta^*/h\}} .
	\end{align*}
	We study separately the first two terms of the above r.h.s. If $y^n_k <\theta_*h$, Proposition \ref{propjumps} gives $|y^{n+1}_{k_u}-y^n_k| \leq C_*h $ and $|y^{n+1}_{k_d}-y^n_k| \leq C_*h $
	so that
	$$
	\E[( Y^h_{n+1}-Y^h_n)^2\mid Y^{h}_n=y^n_k]\I_{\{ y^n_k \leq \theta_*h\}}
	=\varphi_1(y^n_k)h^2\I_{\{y^n_k \leq \theta_*h\}},
	$$
	with $\varphi_1$ such that
	$
	|\varphi_1(y)|\leq C_*^{2}.
	$
	If instead  $\theta_*h\leq y^n_k\leq \theta^*/h$, by using \eqref{proba} we get
	\begin{align*}
	(y^{n+1}_{k_u}-y^n_k)^2p_u(n,k)+ (y^{n+1}_{k_d}-y^n_k)^2p_d(n,k)
	=\sigma^2y^n_kh+\frac{\sigma^2}2\Big( \kappa(\theta-y^n_k)-\frac{\sigma^2}8 \Big)h^2.
	\end{align*}
	So,
	$$
	\E[( Y^h_{n+1}-Y^h_n)^2\mid Y^{h}_n=y^n_k]\I_{\{\theta_*h\leq y^n_k \leq \theta^*/h\}}
	=\big(\sigma^2y^n_kh+\varphi_2(y^n_k)h^2\big)\I_{\{\theta_*h\leq y^n_k \leq \theta^*/h\}},
	$$
	with $\varphi_2$ such that
	$
	|\varphi_2(y)|\leq \frac{\sigma^2}2\Big( \kappa(\theta+y)+\frac{\sigma^2}8 \Big).
	$
	By inserting, \eqref{momento2-gen} follows with $g_h$ satisfying
	$$
	|g_h(Y^h_n)|\leq c_1(1+Y^h_n)h^2+
	\E((Y_{n+1}^h-Y_n^h)^2+\sigma h Y_n^h\mid Y^h_n)\I_{\{Y^h_n\geq \theta^*/h\}},
	$$
$c_1$ denoting a suitable constant. By Proposition \ref{moments-2} and the Markov inequality, \eqref{stima-g}  follows.	

	Finally, for \eqref{momento3-gen}, we write
	\begin{align*}
	&\E[( Y^h_{n+1}-Y^h_n)^3\mid Y^h_n=y^n_k]  
	=\E[( Y^h_{n+1}-Y^h_n)^3\mid Y^h_n=y^n_k]\I_{\{y^n_k \leq \theta_*h\}}\\
	&\quad +\E[( Y^h_{n+1}-Y^h_n)^3\mid Y^h_n=y^n_k]\I_{\{\theta_*h<y^n_k < \theta^*/h\}} 
	+\E[( Y^h_{n+1}-Y^h_n)^3\mid Y^h_n=y^n_k] \I_{\{y^n_k \geq \theta^*/h\}} .
	\end{align*}
	Now, if $y^n_k \leq \theta_*h$ then $|Y^h_{n+1}-y^n_k|^3\leq C_*^3 h^3$. If instead $\theta_*h< y^n_k < \theta^*/h$, by \eqref{proba} one obtains
	\begin{align*}
	(y^{n+1}_{k_u}-y^n_k)^3p_u(n,k)+ (y^{n+1}_{k_d}-y^n_k)^3p_d(n,k)
	=\mu_Y(y^n_k)h^{2}\Big(\sigma^{2}y^n_k+\frac{3\sigma^{4}}{16} \,h\Big)+\Big(\frac{\sigma^{4}}{2}\,y^n_k+\frac{\sigma^{4}}{16}\,h\Big)h^{2}.
	\end{align*} 
Therefore,
$$
	|j_h(Y^h_n)|\leq c_2 h^2(1+(Y^h_n)^2)+\E(|Y_{n+1}^h-Y_n^h|^3+\sigma h Y_n^h\mid Y^h_n)\I_{\{Y^h_n\geq \theta^*/h\}},
$$
$c_2$ denoting a suitable constant, and again by Proposition \ref{moments-2} and the Markov inequality, \eqref{stima-j} follows.	
\end{proof}

We are finally ready for the
\begin{proof}[Proof of Theorem \ref{conv_T-CIR}]
	By Theorem 4.1 in \cite{A-MC} (or Corollary \ref{corollary-cir}), one has that if $f\in C^{4}_\pol(\R_+)$  then  $u\in C^{4}_{\pol,T}(\R_+)$. Since Assumption $\mathcal{H}_1$ and $\mathcal{H}_2$ both hold,
	the statement  follows as an application of Theorem \ref{conv_T}.
\end{proof}

	\section{Hybrid schemes for jump-diffusions and convergence rate   } \label{sect-hybrid}
We now introduce a $m$-dimensional jump-diffusion $(X_t)_{t\in[0,T]}$ whose dynamics is given by coefficients depending on  the process $(Y_t)_{t\in[0,T]}$ discussed in  Section \ref{sect-markovapprox}. More precisely, we consider  the stochastic system
\begin{equation}\label{generalsystem}
\begin{cases}
dX_t= \mu_X(Y_t)dt+\sigma_X(Y_t)\, dB_t+\gamma_X(Y_t)dH_t,  \qquad & X_0 \in \R^m,\\
dY_t=\mu_Y(Y_t)dt+\sigma_Y(Y_t)\,dW_t, \qquad &Y_0\in\D,
\end{cases}
\end{equation}
where  $B$ is a $\ell_1$-dimensional Brownian motion independent of $W$ and  $H$ is a $\ell_2$- dimensional compound Poisson process with intensity  $\lambda$ and i.i.d. jumps  $\{J_k\}_k$ taking values in $\R^{\ell_2}$, that is,
\begin{equation}\label{H}
H_t=\sum_{k=1}^{K_t} J_k,
\end{equation}
$K$ denoting a Poisson process with intensity $\lambda$.  We assume that the Poisson process $K$, the jump amplitudes $\{J_k\}_k$ and the Brownian motion $(B,W)$ are independent. Moreover, we ask that $J_1$ has a density $p_{J_1}$, so that the L\'evy measure associated with $H$  has a density as well:
\begin{equation*}
\nu(dx)=\nu(x)dx=\lambda p_{J_1}(x)dx.
\end{equation*}
We denote by  $\mathcal{L}$  the infinitesimal generator associated with the diffusion pair $(X,Y)$:
\begin{equation}\label{general-L}
\begin{array}{rl}
\L g(x,y)=&
\frac 12 \mbox{Tr}(a(y)D^2_{x,y}g(x,y))+\mu(y) \cdot \nabla_{x,y}g(x,y) \\
&\displaystyle
+ \int (g(x+\gamma_X(y)\zeta,y)-g(x,y))\nu(d\zeta),
\end{array}
\end{equation}
where $\mu(y)=(\mu_X(y),\mu_Y(y))^\star$ and $a(y)=\sigma\sigma^\star(y)$, where 
$$
\sigma(y)=\begin{pmatrix}
\sigma_X(y)&0_{m\times d}\\
0_{d\times m}&\sigma_Y(y)
\end{pmatrix}.
$$ Here, $D^2_{x,y}$ and $\nabla_{x,y}$ are respectively the Hessian and the gradient operator w.r.t. the space variables $(x,y)$.
We    assume that the coefficients of $X$ do not depend on the time variable just to simplify the notation, but  all the proofs in this paper are still valid in the time-depending case under  non restrictive classical assumptions.  

Hereafter, we fix $T>0$ , $f:\R^m\times \mathcal{D}\rightarrow \R$ and we define
\begin{equation}\label{european}
u(t,x,y)=\E\Big[f(X^{t,x,y}_T,Y^{t,y}_T)\Big], \qquad(t,x,y)\in[0,T]\times \R^m\times\D,
\end{equation}
where $(X^{t,x,y}_s,Y^{t,x}_s)_{s\in [t,T]}$ is the solution of \eqref{generalsystem}  with starting condition $(X_t,Y_t)=(x,y)$.
We do not enter in specific assumptions but from now on, the following requests (1) and (2)  are assumed to hold:
\begin{itemize}
	\item [(1)]
	there exists a unique weak solution of \eqref{generalsystem}  and $\P((X_t,Y_t)\in\R^m\times \mathcal{D}\ \forall t)=1$;
	\item[(2)]
	the function $u$ in \eqref{european} solves the PIDE
	\begin{equation}\label{PDEB}
	\left\{
	\begin{array}{ll}
	\partial_tu(t,x,y)+\mathcal{L}u(t,x,y)=0, \quad &(t,x,y)\in [0,T)\times\R^m\times \D, \\
	u(T,x,y)=f(x,y),&\mbox{ in } \R^m\times\D ,
	\end{array}
	\right.
	\end{equation}
	$\mathcal{L}$ being given in \eqref{general-L}.
\end{itemize} 

We are interested in computing $u(0,X_0,Y_0)=\E\big[f(X_T,Y_T)\big]$.
This is a problem of interest in a large number of applications. For example, in finance $X$ can represent the asset log-price (or a transformation of it) and $Y$ can be interpreted as a random source such as  a stochastic volatility and/or a stochastic  interest rate, so   $u(t,x,y)$ represents the  value function at time $t$ of a European option with maturity $T$ and  (discounted) payoff $f$. In next  Section \ref{sect-bates} we give an application to the Heston model \cite{heston} and the  Bates model \cite{bates}.

	\subsection{The hybrid procedure}\label{sect-hybrid2}
	
Let $u$ be given in \eqref{european}.
We study here the  computation of $u(0,X_0,Y_0)$ by a backward hybrid procedure developed in \cite{bcz,bcz-hhw,bctz}. Roughly speaking, one uses a Markov  chain in order to approximate the  process $Y$ and a different numerical procedure to handle the jump-diffusion component $X$. Let us briefly recall the main ideas and describe the approximation of $u$.

We start from the representation of $u(t,x,y)$ at times $nh$, $h=T/N$ and $n=0,\ldots,N$,  by the usual 
dynamic programming principle: for $(x,y)\in\R^m\times\D$,
\begin{equation} \label{backward}
\begin{cases}
u(T,x,y)= f(x,y)\quad
\mbox{and as } n=N-1,\ldots,0,\\
u(nh,x,y) =  
\E\Big[u\big((n+1)h, X_{(n+1)h}^{nh,x,y}, Y_{(n+1)h}^{nh,y}\big)
\Big].
\end{cases}
\end{equation}
So, the central issue is to have a good approximation of the expectations in \eqref{backward}.

As a first step, let $( \hat{Y}^h_{n})_{n=0,\ldots,N}$  be a Markov chain 
which approximates  $Y$. Of course, we assume that $( \hat{Y}^h_{n})_{n=0,\ldots,N}$ is  independent of the noises  $B$ (Brownian motion) and $H$ (compound Poisson process) driving $X$ in \eqref{generalsystem}.  Then, at each step  $n=0,1,\ldots,N-1$, for every $y\in \mathcal{Y}^h_n\subseteq \D$ (the state space of $\hat{Y}^h_{n}$) we write
\begin{align*}
\E\Big[u\big((n+1)h,  X_{(n+1)h}^{nh,x,y},  Y_{(n+1)h}^{nh,y}\big)\Big] \approx 	\E\Big[u\big((n+1)h,  X_{(n+1)h}^{nh,x,y},  \hat{Y}^h_{n+1}\big)\big|\hat{Y}^h_n=y\Big].
\end{align*}

As a second step, we approximate  the component  $X$ on $[nh,(n+1)h]$ by freezing the coefficients in \eqref{generalsystem} at the observed position $\hat{Y}^h_n=y$, that is, for $t\in[nh,(n+1)h]$,
$$
X_{t}^{nh,x,y}\stackrel{\mbox{\tiny law}}{\approx} \widehat{X}^{nh,x}_t(y)=x+\mu_X(y)(t-nh)+\sigma_X(y)\, (B_t-B_{nh})+\gamma_X(y)(H_t-H_{nh}).
$$
Therefore,	by using that the Markov chain, $B$ and $H$ are all independent, we write
\begin{align*}
\E\Big[u\big((n+1)h,  X_{(n+1)h}^{nh,x,y},  Y_{(n+1)h}^{nh,y}\big)\Big] 
&\approx 
\E\Big[u\big((n+1)h,  \widehat{X}^{nh,x}_{(n+1)h}(y),  \hat{Y}^h_{n+1}\big)\big|\hat{Y}^h_n=y\Big]\\
&=\E\big[ \phi(  \hat{Y}^{h}_{n+1};x,y )\big|\hat{Y}^h_n=y\big],
\end{align*}
where
\begin{equation}
\label{phi}
\phi(\zeta;x,y)= \E\big[u((n+1)h, \widehat{X}^{nh,x}_{(n+1)h}(y),\zeta)\big].
\end{equation}
From the Feynman-Kac formula, one gets $\phi(\zeta;x,y)=v(nh,x;y,\zeta)$, where $(t,x)\mapsto v(t,x;y,\zeta)$ is the solution at time $nh$ of the parabolic PIDE Cauchy problem
\begin{equation}\label{PDE-barug-h}
\begin{array}{ll}
\displaystyle
\partial_t v+\L^{(y)}v=0, \qquad & \mbox{in } [nh,(n+1)h)\times \R^m, \\
\displaystyle
v((n+1)h,x;y,\zeta)
=u((n+1)h,x,\zeta), & x\in \R^m,
\end{array}
\end{equation}
where
\begin{equation}
\label{genh}
\L^{(y)}g(x)=\mu_X(y)\cdot\nabla_xg(x) +\frac 12\mbox{ Tr}(a_X(y)D^2_{x}g(x))+ \int \big(g(x+\gamma_X(y)\zeta)-g(x)\big)	\nu(\zeta)d\zeta
\end{equation}
is an integro-differential operator, acting on the functions $g=g(x)$. Here $a_X(y)=\sigma_X(y)\sigma_X^\star(y)$, while
$\nabla_x$ and $D^2_x$ are, respectively,  the  gradient vector and  the Hessian matrix with respect to  $x\in\R^m$. Recall that in \eqref{PDE-barug-h}--\eqref{genh}, $y \in \D$ is just a parameter, so  $\L^{(y)}$ has constant coefficients.

Consider now a numerical solution of the PIDE \eqref{PDE-barug-h}. Let $\dx=(\dx_1,\dots,\dx_m)$ denote a fixed spatial step and set $\mathcal{X}$ a grid on $\R^m$ given by $\mathcal{X}=\{x\,:\,x=((X_0)_1+i_1\Delta x_1,\dots, (X_0)_m+i_m\Delta x_m), (i_1,\ldots,i_m)\in \Z^m\}$. For $y\in\mathcal{D}$,  let $\Pi^h_{\dx}(y)$ be a linear operator (acting on suitable functions on $\mathcal{X}$) which gives 
the approximating solution to the PIDE \eqref{PDE-barug-h} at time $nh$.  Then, as $x\in\mathcal{X}$, we get the numerical approximation
\begin{align*}
\E\Big[u\big((n+1)h,  X_{(n+1)h}^{nh,x,y},  Y_{(n+1)h}^{nh,y}\big)\Big] 
&\approx 
\E\Big[\Pi^h_{\dx}(y)u\big((n+1)h, \cdot,  \hat{Y}^h_{n+1}\big)(x)\big|\hat{Y}^h_n=y\Big].
\end{align*}
Therefore, by inserting  in \eqref{backward}, the hybrid numerical procedure works as follows:
the function $x\mapsto u(0,x,Y_0)$, $x\in\mathcal{X}$,  is approximated by $u^h_0(x,Y_0)$ backwardly defined as
\begin{equation}\label{backward-ter0}
\begin{cases}
u^h_N(x,y)= f(x,y),\quad \mbox{$(x,y)\in \mathcal{X}\times\mathcal{Y}^h_N$},
\quad \mbox{and as $n=N-1,\ldots,0$:}\\
u^h_n(x,y) =  \E[
\Pi^h_{\dx}(y) u^h_{n+1}( \cdot,  \hat{Y}^h_{n+1})(x)\mid \hat{Y}^h_n=y],\quad \mbox{$(x,y)\in \mathcal{X}\times\mathcal{Y}^h_n$.}
\end{cases}
\end{equation}

	\subsection{Convergence speed of the scheme \eqref{backward-ter0}} 
	\label{sect-convergence}

We introduce the following assumption on the linear operator $\Pi^h_{\dx}(y)$ in \eqref{backward-ter0} (recall the notation $l_p(\mathcal{X})$ in  Section \ref{sect-notation}).

\medskip

\noindent
\textbf{Assumption $\mathcal{K}(p,c,\mathcal{E})$.} 
\textit{Let $p\in [1,\infty]$, $c=c(y)\geq 0$, $y\in\D$ and $\mathcal{E}= \mathcal{E}(h,\dx)\geq 0$ such that 
	$\lim_{(h,\dx)\rightarrow 0}\mathcal{E}(h,\dx)=0.$
	We say that the 
	linear operator $\Pi^h_{\dx}(y):l_p(\mathcal{X})\to l_p(\mathcal{X})$, $y\in\D$, satisfies Assumption $\mathcal{K}(p,c,\mathcal E)$ if 
	\begin{equation}\label{stab}
	| \Pi^h_{\dx}(y)|_p\leq1+c(y)h
	\end{equation}
	and, $u$ being defined in \eqref{european}, for every $n=0,\dots, N-1$, one has
	\begin{equation}\label{cons}
	\E\Big[	\Pi^h_{\dx}(\hat{Y}^h_n) u((n+1)h,\cdot,\hat{Y}^h_{n+1})(x)\,\big|\, \hat{Y}^h_n\Big]= u(nh,x,\hat{Y}^h_n)+ \mathcal{R}_n^h(x,\hat{Y}^h_n),
	\end{equation}
	where the remainder $\mathcal{R}_n^h(x ,\hat{Y}^h_n)$ satisfies the following property: there exist $\bar h, C>0$ such that for every $h<\bar h$, $\dx<1$ and $n\leq N=\lfloor T/h\rfloor$ one has		
	\begin{equation}\label{ass_errore}
	\Big\|e^{\sum_{l=1}^{n}\, c(\hat{Y}^h_l)h} |\mathcal{R}_n^h(\cdot,\hat{Y}^h_n)|_p\Big\|_p\leq C h\mathcal{E}(h,\dx), \qquad \mbox{if } p\in [1,\infty].
	\end{equation}
}

Assumption $\mathcal{K}(p,c,\mathcal{E})$ is inspired by the Lax-Richtmeyer's convergence theorem \cite{Lax}. In fact, recall that the  numerical procedure \eqref{backward-ter0} aims to solve the multidimensional equation 
$$
\partial_tu(t,x,y)+\mathcal{L}u(t,x,y)=0, \quad (t,x,y)\in [0,T)\times\R^m\times \D.
$$
Being dependent on $y$, the coefficients of  $\L$ (see \eqref{general-L})  are not constant as required by the Lax-Richtmeyer's result.  But at each time step $n$, the hybrid scheme isolates the component $y$ and applies the discrete operator $\Pi^h_{\dx}(y)$ to numerically solve the PIDE 
$$	
\partial_t v(t,x)+\L^{(y)}  v(t,x)=0, \qquad (t,x)\in [nh,(n+1)h)\times \R^m.
$$ 
Here, $y$ is just a parameter (the current position of the Markov chain), so the coefficients of $\L^{(y)}$	(see \eqref{genh}) are indeed constant. That's  why the Lax-Richtmeyer technique can be adapted, as it follows in the next result.

\begin{theorem}\label{convergencebates}
	Assume that $\Pi^h_{\dx}(y)$, $y\in\D$, satisfies Assumption $\mathcal{K}(p,c,\mathcal{E})$. Let $u$ be defined  in \eqref{european}   and $u^h$ be the approximation through the  scheme \eqref{backward-ter0}.
	Then, there exist $\bar h, C>0$ such that for every $h<\bar h$ and $\dx<1$  one has
	\begin{equation}
	| u(0,\cdot,Y_0)-u^h_0(\cdot,Y_0)|_{p} \leq CT\mathcal{E}(h,\dx). 
	\end{equation}
\end{theorem}
\begin{proof}
	Let  $\mathrm{err}^{h}_n (\cdot,\hat{Y}^h_n)$ be the  error at time $nh$, defined by
	$$
	\mathrm{err}^{h}_n(\cdot,\hat{Y}^h_n)=u(nh,\cdot,\hat{Y}^h_n)-u^h_n(\cdot,\hat{Y}^h_n).
	$$
	Note that $\mathrm{err}_{N}^h(\cdot,\hat{Y}_N^h)=0$, because the final condition is the same. By \eqref{cons} and \eqref{backward-ter0}, we can write
	\begin{align*}
	\mathrm{err}^h_n (\cdot,\hat{Y}^h_n)&=  \E[\Pi^h_{\dx}(\hat{Y}^h_n)\mathrm{err}^h_{n+1}(\cdot,\hat{Y}^h_{n+1}) |\hat{Y}^h_n ] -\mathcal{R}_n^h(\cdot ,\hat{Y}^h_n)
	\end{align*}
	and, by iterating,
	$$
	\mathrm{err}^h_0(\cdot,Y_0)= -\sum_{n=0}^{N-1}\E\Big[ \Big(\prod_{l=0}^{n-1} \Pi^h_{\dx}(\hat{Y}^h_l))\Big) \mathcal{R}_n^h(\cdot,\hat{Y}^h_n)\Big],
	$$
	in which we use the convention $ \prod_{l=0}^{-1}(\cdot)=\Id$. We use now \eqref{ass_errore}. For $p\neq \infty$,
	$$
	\begin{array}{l}
	|\mathrm{err}_h^0(\cdot,Y_0)|_{p}
	\leq \sum_{n=0}^{N-1}\Big|\E\Big[ \Big(\prod_{l=0}^{n-1} \Pi^h_{\dx}(\hat{Y}^h_l)\Big) \mathcal{R}_n^h(\cdot,\hat{Y}^h_n)\Big]\Big|_{p}\\
	\leq \sum_{n=0}^{N-1}\E\Big[ \Big|\Big(\prod_{l=0}^{n-1} \Pi^h_{\dx}(\hat{Y}^h_l)\Big) \mathcal{R}_n^h(\cdot,\hat{Y}^h_n)\Big|^p_{p}\Big]^{1/p}\\
	\leq  \sum_{n=0}^{N-1}	\left(	\E\big[e^{\sum_{l=1}^{n}pc(\hat{Y}^h_l)h} |\mathcal{R}_n^h(\cdot,\hat{Y}^h_n)|^p_{p}\big]\right)^{\frac 1 p}\\
	\leq \sum_{n=0}^{N-1}hC\mathcal{E}(h,\dx)
	\leq TC\mathcal{E}(h,\dx).
	\end{array}
	$$
	The case  $p=\infty$ follows the same lines.
\end{proof}
	
	\subsection{ An application: finite difference schemes}\label{sect-finitedifference}	

We specify here some settings ensuring  that the assumptions of  Theorem \ref{convergencebates} are satisfied. In particular, we choose the operator $\Pi^h_{\dx}(y)$ in \eqref{backward-ter0} by means of two different finite difference schemes: the first one allows us to study the convergence in the $l_2$-norm (Section \ref{sect-l2}), while the second one in the $l_\infty$-norm (Section \ref{sect-linf}). For the sake of readability, we consider the case $m=d=\ell=\ell_1=\ell_2=1$. Moreover, hereafter we assume that the coefficients in \eqref{generalsystem} satisfy:
\begin{itemize}
	\item [(a)]
	$\mu=(\mu_X,\mu_Y)^\star$ and $\sigma_X$ have polynomial growth; 
	\item[(b)]
	either $\gamma_X\equiv 0$ (no jumps) or there exists $\varepsilon>0$ such that  $\inf_{y\in\mathcal{D}}|\gamma_X(y)|\geq \varepsilon$ (uniform ellipticity condition). 
\end{itemize} 
Let us stress that (a) is necessary in order to control suitable remaining terms, whereas (b) follows from an appropriate change of variable allowing to set up the quadrature rules (see Remark \ref{rem-gamma} below). In particular, (b) allows us to define the measure $\nu_y$ as follows:
\begin{equation}\label{nuy}
\nu_y(x)=\left\{
\begin{array}{ll}
0 &\mbox{if } \gamma_X\equiv 0,\\
\frac 1{|\gamma_X(y)|}\nu(\frac{x}{\gamma_X(y)}) & \mbox{otherwise},
\end{array}
\right.
\quad y\in\mathcal{D},
\end{equation}
$\nu$ denoting the  density of the L\'evy measure.

\begin{proposition}\label{prop-nu}
	If $\frac{\nu'}{\nu},\frac{\nu''}{\nu}\in L^1(\R,d\nu)$, there exists $c_\nu\geq 0$ such that
	\begin{equation}\label{cnu}
	\sum_{l\in\Z}\nu_y(l\dx)\dx\leq \lambda c_\nu,\quad \forall y\in\mathcal{D}.
	\end{equation}
\end{proposition}

\begin{proof}
	The proof follows from the technical Lemma \ref{lemma-poisson} below: if $\gamma_X$ is non null, $(i)$ applied to $g(x)=\nu_y(x)$ gives
	$
	\sum_{l\in \Z}\nu_y(l\dx)\dx\leq \int_\R \nu(x)dx+\frac{|\dx|^2}{12|\gamma_X(y)|^2}\int_\R |\nu''(x)|dx.
	$
	Now we use the ``uniformity'' condition $\inf_{y\in\mathcal{D}}|\gamma_X(y)|\geq \varepsilon$, and the statement holds. 
\end{proof}

\begin{lemma}\label{lemma-poisson} Let $g\in C^2(\R)$.
	
	\noindent
	$(i)$ 
	If $g,g',g''\in L^1(\R,dx)$ then 
	\begin{equation}\label{poisson1}
	\Big|\sum_{l\in \Z}g(l\Delta x )\Delta x -\int_\R g(x)dx\Big|\leq \frac {\Delta x ^2}{12} \,|g''|_{L^1(\R,dx)}.
	\end{equation}
	
	\noindent
	$(ii)$ 
	If $g,g', g''\in L^2(\R,dx)$ then
	\begin{equation}\label{poisson2}
	|g|_2^2
	\leq |g|_{L^2(\R,dx)}^2+
	\frac {\Delta x ^2}{6} \,\big(|g'|_{L^2(\R,dx)}^2+|g|_{L^2(\R,dx)}\times|g''|_{L^2(\R,dx)}\big).
	\end{equation}
\end{lemma}	

\begin{proof}
	We first recall the Poisson summation formula. It is worldwide famous but is usually written on the Schwartz space, we use here the following version (Appendix \ref{SU-Poisson} contains the detailed proof):  if $\varphi\in C^2(\R)$ with $\varphi,\varphi',\varphi''\in L^1(\R,dx)$ then
	\begin{equation}\label{poisson-formula}
	\sum_{n\in \Z}\varphi(n)=\int_{\R} \varphi(x)dx+\sum_{n\in\Z,n\neq 0}\int_{\R}\varphi(x)e^{-2\pi \ii n x}dx.
	\end{equation}
	
	$(i)$ 
	We apply \eqref{poisson-formula} to $\varphi(x)=g(x\Delta x )$. So,
	$$
	\begin{array}{l}
	\sum_{n\in \Z}g(n\Delta x )\Delta x -\int_{\R} g(x)dx=
	\sum_{n\in\Z,n\neq 0} 
	\int_{\R}g(x)e^{-2\pi \ii n x/\Delta x  }dx\\
	\ \ =\sum_{n\in\Z,n\neq 0} 
	\frac{\Delta x ^2}{(2\pi \ii n)^2}\int_{\R}g''(x)e^{-2\pi \ii n x/\Delta x  } dx,
	\end{array}
	$$
	the latter inequality coming from the integration by parts formula. The statement holds by recalling that $\sum_{n\geq 1}\frac{1}{n^2}=\frac{\pi^2}{6}$.
	
	\smallskip
	
	\noindent
	$(ii)$
	\eqref{poisson2} immediately follows by applying \eqref{poisson1} to the function $x\mapsto g^2(X_0+x)$. This statement  will be used to handle the error in $l_2$-norm coming from suitable Taylor's expansions and from the quadrature approximation.  
	%
\end{proof}


\subsubsection{Convergence in $l_2$-norm}\label{sect-l2}
	
We study here the hybrid procedure introduced in \cite{bctz} for the Bates model.
Recall that, for $y\in\D$, $\Pi^h_{\dx}(y)$ gives the numerical solution on $\mathcal{X}=\{x_i=X_0+i\Delta x\}_{i\in \Z}$ a time $nh$ to the PIDE \eqref{PDE-barug-h}, the operator $\L^{(y)}$ therein being given in \eqref{genh}. It is clear that the solution $v$ of \eqref{PDE-barug-h} depends on $y$ and $\zeta$ as well, but these are just parameters (and not variables of the PIDE), so for simplicity we drop here such dependence.
So, we split the operator $	\L^{(y)}v(t,x)=\L_{\mbox{{\tiny diff}}}^{(y)}v(t,x)+\L_{\mbox{{\tiny int}}}^{(y)}v(t,x)$ in its differential and integral part:
\begin{align}
&\label{L_d}
\L_{\mbox{{\tiny diff}}}^{(y)}v(t,x)=\mu_X(y)\partial_x v(t,x) +\frac 12\sigma_X^2(y)\partial^2_x v(t,x),\\
\label{L_i}
&\L_{\mbox{{\tiny int}}}^{(y)}v(t,x)
=\int \big(v(t,x+\gamma_X(y)z)-v(t,x)\big)	\nu(z)dz.
\end{align} 
%
%
%
%
We use the central finite difference scheme to solve $\L^{(y)}_{\mbox{{\tiny diff}}}v$ and the trapezoidal rule in order to approximate the integral term $\L^{(y)}_{\mbox{{\tiny int}}}v$. Applying an implicit-explicit method in time, we obtain an approximating solution $v^{n}=(v^n_j)_{j\in\Z}\,:\,\mathcal{X}\to \R$ to the PIDE \eqref{PDE-barug-h} given by
\begin{equation}\label{equaz}
A^h_{\dx}(y)v^n=B^h_{\dx}(y)v^{n+1},
\end{equation}
where the linear operators $A^h_{\dx}(y)$ 
is defined as
\begin{equation}\label{A}
(A^h_{\dx})_{ij}(y)=\begin{cases}
\alpha^h_{\dx}(y)-\beta^h_{\dx}(y),\qquad &\mbox{ if }i=j+1,\\
1+2\beta^h_{\dx}(y),\qquad &\mbox{ if }i=j,\\
-\alpha^h_{\dx}(y)-\beta^h_{\dx}(y),\qquad &\mbox{ if }i=j-1,\\
0 &\mbox{ if }|i-j|>1
\end{cases},
\end{equation} 
with
\begin{equation}\label{alpha-beta}
\alpha^h_{\dx}(y)=\frac{h}{2\dx} \mu_X(y), \qquad \beta^h_{\dx}(y)=\frac{h}{2\dx^2}\sigma_X^2(y).
\end{equation}
Moreover, we choose the approximation  $B^h_{\dx}(y)$ for $\L_{\mbox{{\tiny int}}}$ in order to work on the same numerical  grid $\mathcal{X}$. This is achievable by using  the change of variable in $\L^{(y)}_{\mbox{{\tiny int}}}$:
$\L_{\mbox{{\tiny int}}}^{(y)}v(t,x)
=\int \big(v(t,x+\zeta)-v(t,x)\big)	\nu_y(\zeta)d\zeta$, $\nu_y$ being defined in \eqref{nuy}. Then we get
\begin{equation}\label{B}
(B^h_{\dx})_{ij}(y)=\begin{cases}
h\nu_y((j-i)\dx)\dx & \mbox{ if }j\neq i,\\
1+h\Big(\nu_y(0)\dx-\sum_{l\in\Z} \nu_y(l\dx)\dx\Big)	&\mbox{ if }i=j .
\end{cases}
\end{equation}
Note that $B^h_{\dx}(y)=\Id$ if $\gamma_X\equiv 0$.

\begin{remark}\label{rem-gamma} The above construction for $B^h_{\dx}(y)$ justifies the ``uniformly ellipticity'' requirement for $\gamma_X$. One could drop this assumption by avoiding the change of variable in $\L_{\mbox{{\tiny int}}}$. But this would bring to the use of a numerical grid depending on $y$ and therefore, the introduction of suitable interpolations. As a consequence, one would have a complication of the numerical scheme, the introduction of technical details and further notations. 
\end{remark}

The operators $A^h_{\dx}(y)$ and $B^h_{\dx}(y)$ in \eqref{equaz} and \eqref{B} respectively, have the following properties.
\begin{lemma}\label{norma}
	For every  $y\in \D$, $A^h_{\dx}(y):l_2(\mathcal{X})\rightarrow l_2(\mathcal{X})$ is  invertible and moreover, 
	$$\sup_{y\in\mathcal{D}}| (A^h_{\dx})^{-1}(y)|_2\leq1.
	$$
	And if $\frac{\nu'}{\nu},\frac{\nu''}{\nu}\in L^1(\R,d\nu)$ then $\sup_{y\in\mathcal{D}}|B^h_{\dx}(y)|_2$ $\leq 1+2\lambda c_\nu h$, $c_\nu$ being defined in \eqref{cnu}.
\end{lemma}
\begin{proof}
	Let $ \mathcal{F}\,:\,  l_2(\mathcal{X})\to L^2([0,2\pi), dx)$ denote the Fourier transform: 	$ \mathcal{F}(\varphi) (\theta) = \frac{\dx}{\sqrt{2\pi}}\sum_{
		j\in \Z} \varphi_j e^{-\ii j\Delta x\theta}$, $\theta\in [0,2\pi)$, $\varphi\in l_2(\mathcal{X})$. 
	
	Fix $y\in \D $ and $w \in l_2(\mathcal{X})$. $v\in  l_2(\mathcal{X})$ satisfies  $A^h_{\dx}(y)v=w$ iff 
	$ \mathcal{F}(A^h_{\dx}(y)v)= \mathcal{F}(w)$. Straightforward computations give (see e.g. the proof of Theorem 5.1 in \cite{bctz})
	$ \mathcal{F}(A^h_{\dx}(y)v)=\psi \times   \mathcal{F}(v)$, with $\psi(\theta)=(\alpha^h_{\dx}(y)-\beta^h_{\dx}(y))e^{-\ii\theta\dx}+1+2\beta^h_{\dx}(y)-(\alpha^h_{\dx}(y)+\beta^h_{\dx}(y))e^{\ii\theta\dx}$. 
	It can be easily seen that $|\psi(\theta)|\geq 1+2\beta^h_{\dx}(y)(1-\cos(\theta\dx))$ $
	\geq 1$. Hence $ \mathcal{F}(v) = \mathcal{F} (w)/\psi\in L^2([0,2\pi),dx)$ and its inverse Fourier transform uniquely defines the solution $v\in l_2(\mathcal{X})$ to $A^h_{\Delta x}(y)v=w$.
	Thus $A^h_{\dx}$ is invertible. Moreover, we obtain 
	$	|  \mathcal{F}(v)|_{L^2([0,2\pi),dx)}$ $\leq | \mathcal{F} (w)|_{L^2([0,2\pi),dx)}$.
	By the Parseval identity 
	we get $| (A^h_{\dx})^{-1}(y)w|_2 \leq 	| w|_2$, so $| (A^h_{\dx})^{-1}(y)|_2\leq 1$.
	Finally, for $w\in l_2(\mathcal{X})$ straightforward computations give
	$$
	\mathcal{F}(B^h_{\dx}(y)w)(\theta)=\Big( 1+ h\dx\sum_l  \nu_y(l\dx) (e^{\ii l\theta} -  1)       \Big) \mathcal{F} (w) (\theta).
	$$
	Then, $| \mathcal{F}(B^h_{\dx}(y)w) |_{L^2([0,2\pi),dx)}\leq ( 1+ 2\lambda c_\nu h)| \mathcal{F}(w)|_{L^2([0,2\pi),dx)}$ because  \eqref{cnu} holds. By the Parseval relation, $|B^h_{\dx}(y)w|_{2}\leq ( 1+ 2\lambda c_\nu h)|w|_{2}$, which concludes the proof.
\end{proof}

We can now state the convergence result in $l_2(\mathcal{X})$, saying that the rate of convergence is of  the second order in space, because of the choice of a second order finite difference scheme, and of first order in time, as it is natural also for the presence of the approximating Markov chain $\hat{Y}^h$ (see   Theorem \ref{conv_T}).
\begin{theorem}
	\label{conv-H}
	
	Let $u$ be defined in \eqref{european} and $(u^h_n)_{n=0,\ldots,N}$ be given by \eqref{backward-ter0} with the choice
	$$
	\Pi^h_{\dx}(y)=(A^h_{\dx})^{-1}B^h_{\dx}(y),
	$$
	$A^h_{\dx}(y)$ and $B^h_{\dx}(y)$ being given in \eqref{A} and \eqref{B} respectively. Assume that
	\begin{itemize}
		\item $\frac{\nu'}\nu,\frac{\nu''}\nu\in L^2(\R,d\nu)$;
		\item the Markov chain $(\hat{Y}^h_n)_{n=0,\dots, N}$ satisfies assumptions $\mathcal{H}_1$ and $\mathcal{H}_2$;
		\item $u\in  C^{2,6}_{\pol, T}( \R, \D)$.
	\end{itemize}
	Then, there exist $\bar h,C>0$ such that for every $h<\bar h$ and $\dx<1$ one has
	\begin{equation}\label{conv-1}
	| u(0,\cdot,Y_0)-u^h_{0}(\cdot,Y_0)|_{2} \leq CT(h+\dx^2).
	\end{equation}
\end{theorem}

\begin{proof}
The proof follows from   Theorem \ref{convergencebates} once we prove that Assumption $\mathcal{K}(2,2\lambda c_\nu,h+\dx^2)$ holds.

First, Lemma \ref{norma} gives $|\Pi^h_{\dx}(y)|_2\leq |(A^h_{\dx})^{-1}(y)|_2|B^h_{\dx}(y)|_2\leq 1+2\lambda c_\nu h$, so \eqref{stab} holds with $c(y)=2\lambda c_\nu$. We prove now \eqref{ass_errore} with $p=2$ and $\mathcal{E}(h,\dx)=h+\dx^2$. We first note that \eqref{cons} equals to 
\begin{equation}
\label{mmm}
\begin{array}{l}
\E\big[B^h_{\dx}(\hat{Y}^h_n)u((n+1)h, \cdot, \hat{Y}^h_{n+1})(x)\mid \hat{Y}^h_n\big] \\ =A^h_{\dx}(\hat{Y}^h_n)u(nh,\cdot,\hat{Y}^h_n)(x)+A^h_{\dx}(\hat{Y}^h_n)\mathcal{R}^h_n(\cdot,\hat{Y}^h_n)(x).
\end{array}
\end{equation}

\noindent\textbf{Step 1. Taylor expansion of the l.h.s. of \eqref{mmm}.} We set
\begin{equation}\label{Bapplied}
\begin{array}{l}
I_1= B^h_{\dx}(\hat{Y}^h_n)u((n+1)h, \cdot, \hat{Y}^h_{n+1})(x_i) \\
=u((n+1)h,x_i,\hat{Y}^h_{n+1}) \\
+h\sum_l \nu_{\hat{Y}^h_n}(l\dx )\Big(u((n+1)h,x_i+l\dx,\hat{Y}^h_{n+1})-u((n+1)h,x_{i},\hat{Y}^h_{n+1})\Big)\dx.
\end{array}
\end{equation}
In the first term of the above r.h.s. we apply several Taylor's expansion: of $t\mapsto u(t,x_i,\hat{Y}^h_{n+1}) $  around $nh$ up to  order 1,
of $y\mapsto u(nh,x_i,y)$ around $\hat{Y}^h_n$ up to order 3 and of $y\mapsto \partial_tu(nh,x_i,y)$ around $\hat{Y}^h_n$ up to order 1. Rearranging the terms we obtain
$$
\begin{array}{l}
u((n+1)h,x_i,\hat{Y}^h_{n+1})
=u(nh,x_i,\hat{Y}^h_{n}) \\
+ \partial_tu(nh,x_i,\hat{Y}^h_{n})h+\partial_yu(nh,x_i\hat{Y}^h_{n})(\hat{Y}^h_{n+1}-\hat{Y}^h_{n})
+\frac 12 \partial_y^{2}u(nh,x_i,\hat{Y}^h_{n})(\hat{Y}^h_{n+1}-\hat{Y}^h_{n})^{2}\\
+\partial_y\partial_tu(nh,x_i,\hat{Y}^h_{n})\,h(\hat{Y}^h_{n+1}-\hat{Y}^h_{n})
+\frac 16\partial_y^3u(nh,x_i,\hat{Y}^h_{n})(\hat{Y}^h_{n+1}-\hat{Y}^h_{n})^3\\
+R_{1}(n,h,x_i,\hat{Y}^h_{n},\hat{Y}^h_{n+1}),
\end{array}
$$
where $R_1$  is given by
\begin{equation}\label{R1}
\begin{array}{ll}
&	R_{1}(n,h,x_i,\hat{Y}^h_{n},\hat{Y}^h_{n+1})
=h^2\int_0^1(1-\tau)\partial^2_tu(nh+\tau h,x_i,\hat{Y}^h_{n+1})d\tau\\
&	\quad + \frac{(\hat{Y}^h_{n+1}-\hat{Y}^h_n)^4} 6 \int_0^1(1-\zeta)^3\partial^4_yu(nh,x_i,\hat{Y}^h_n+\zeta(\hat{Y}^h_{n+1}-\hat{Y}^h_n))d\zeta\\
&	\quad +h(\hat{Y}^h_{n+1}-\hat{Y}^h_n)^2\int_0^1(1-\zeta)\partial_t\partial^2_yu(nh,x_i,\hat{Y}^h_n+\zeta(\hat{Y}^h_{n+1}-\hat{Y}^h_n))d\zeta.
\end{array}
\end{equation}

For the second term in the r.h.s. of \eqref{Bapplied},  we stop the Taylor expansion of   $t\mapsto u((n+1)h,x_i+l\dx,\hat{Y}^h_{n+1})$ around $nh$ at order 0  and of $y\mapsto u(nh,x_i+l\dx,y)$ around $\hat{Y}^n_h$ at order 1, obtaining
$$
\begin{array}{l}
h\sum_l \nu_{\hat{Y}^h_n}(l\dx)\big[u((n+1)h,x_i+l\dx,\hat{Y}^h_{n+1})-u((n+1)h,x_{i},\hat{Y}^h_{n+1})\big]\dx  \\
=h\sum_l \nu_{\hat{Y}^h_n}(l\Delta x)\big[u(nh,x_i+l\dx,\hat{Y}^h_{n})-u(nh,x_{i},\hat{Y}^h_{n})\big]\dx \\
+h(\hat{Y}^h_{n+1}-\hat{Y}^h_{n})\sum_l \nu_{\hat{Y}^h_n}(\l\dx)\big[\partial_y u(nh,x_i+l\dx,\hat{Y}^h_{n})-\partial_y u(nh,x_{i},\hat{Y}^h_{n})\big]\dx \\
+R_2(n,h,x_i,\hat{Y}^h_{n},\hat{Y}^h_{n+1}),
\end{array}
$$	
where $R_2$ contains the integral terms:
\begin{equation}\label{R2}
\begin{array}{l}
R_2(n,h,x_i,\hat{Y}^h_{n},\hat{Y}^h_{n+1})= 
h^2\sum_l \nu_{\hat{Y}^h_n}(l\Delta x)\dx\times\\
\times \int_0^1(1-\tau)\big[\partial_tu(nh+\tau h,x_i+l\dx,\hat{Y}^h_{n+1})-\partial_tu(nh+\tau h,x_{i},\hat{Y}^h_{n+1})\big]d\tau \\
+h (\hat{Y}^h_{n+1}-\hat{Y}^h_{n})^{2} \sum_l \nu_{\hat{Y}^h_n}(\l\dx)\dx\times \\
\times \!\!\int_0^1\!(1-\zeta)\!\big[\!\partial_yu(nh ,x_i\!+\!l\dx,\hat{Y}^h_n\!+\!\zeta(\hat{Y}^h_{n+1}\!-\!\hat{Y}^h_n))\!-\!\partial_yu(nh,x_{i},\hat{Y}^h_n\!+\!\zeta(\hat{Y}^h_{n+1}\!-\!\hat{Y}^h_n))\!\big]\!d\zeta .
\end{array}
\end{equation}
By resuming, we obtain
\begin{equation}\label{I1}
\begin{array}{rl}
I_1=
&	u(nh,x_i,\hat{Y}^h_{n})
+ \partial_tu(nh,x_i,\hat{Y}^h_{n})h+\partial_yu(nh,x_i,\hat{Y}^h_{n})(\hat{Y}^h_{n+1}-\hat{Y}^h_{n})\\
&+\frac 12 \partial_y^{2}u(nh,x_i,\hat{Y}^h_{n})(\hat{Y}^h_{n+1}-\hat{Y}^h_{n})^{2} \\
&+h\dx\sum_l \nu_{\hat{Y}^h_n}(l\dx )\big[u(nh,x_i+l\dx,\hat{Y}^h_{n})-u(nh,x_{i},\hat{Y}^h_{n})\big]\\
&+\sum_{i=1}^2 R_{i}(n,h,x_i,\hat{Y}^h_{n},\hat{Y}^h_{n+1}) + S(n,h,x_i,\hat{Y}^h_{n},\hat{Y}^h_{n+1}),
\end{array}
\end{equation}
where 
\begin{equation}
\label{S}
\begin{array}{l}
S(n,h,x_i,\hat{Y}^h_{n},\hat{Y}^h_{n+1}) \\
=	\partial_y\partial_tu(nh,x_i,\hat{Y}^h_{n})\,h(\hat{Y}^h_{n+1}-\hat{Y}^h_{n})
+\frac 16\partial_y^3u(nh,x_i,\hat{Y}^h_{n})(\hat{Y}^h_{n+1}-\hat{Y}^h_{n})^3 \\
+h(\hat{Y}^h_{n+1}-\hat{Y}^h_{n})\sum_l \nu_{\hat{Y}^h_n}(l\dx)\big[\partial_yu(nh,x_i+l\dx,\hat{Y}^h_{n})-\partial_yu(nh,x_{i},\hat{Y}^h_{n})\Big]\dx.
\end{array}
\end{equation}
\noindent
\textbf{Step 2. Taylor expansion  of the first addendum in the r.h.s. of \eqref{mmm}}. We set
$$
\begin{array}{rl}
I_2=&A^h_{\dx}u(nh,\cdot,\hat{Y}^h_n)(x_i) \\
=&(\alpha^h_{\dx}(\hat{Y}^h_n)-\beta^h_{\dx}(\hat{Y}^h_n))u(nh,x_{i-1},\hat{Y}^h_n) \\
&+(1+2\beta^h_{\dx}(\hat{Y}^h_n))u(nh,x_{i},\hat{Y}^h_n)
-(\alpha^h_{\dx}(\hat{Y}^h_n)+\beta^h_{\dx}(\hat{Y}^h_n))u(nh,x_{i+1},\hat{Y}^h_n).
\end{array}
$$
We expand with Taylor $x\mapsto u(nh,x,\hat{Y}^h_n)$  around $x_i$ up to order 3
and we insert  the values of $\alpha^h_{\dx}$ and $\beta^h_{\dx}$ in \eqref{alpha-beta}. Rearranging the terms we get 
\begin{equation}\label{I2}
\begin{array}{l}
I_2=
u(nh,x_{i},\hat{Y}^h_n)
-h\mu_X(\hat{Y}^h_n)\partial_xu(nh,x_{i},\hat{Y}^h_n)
-\frac 12 \,h \sigma^2_X(\hat{Y}^h_n)\partial^2_xu(nh,x_{i},\hat{Y}^h_n)\\
\ \ \ +R_3(n,h,x_i,\hat{Y}^h_{n},\hat{Y}^h_{n+1})
\end{array}
\end{equation}
where
\begin{equation}\label{R3}
\begin{array}{l}
R_3(n,h,x_i,\hat{Y}^h_{n},\hat{Y}^h_{n+1})\\
=
\frac{\dx\mu_X(\hat{Y}^h_n)- \sigma_X^2(\hat{Y}^h_n)}{12}\,h\dx^2\int_0^1(1-\eta)^3\partial^4_xu(nh,x_i-\eta\dx,\hat{Y}^h_n)d\eta  \\
-\frac{\dx\mu_X(\hat{Y}^h_n)+ \sigma_X^2(\hat{Y}^h_n)}{12}\,h\dx^2\int_0^1(1-\eta)^3\partial^4_xu(nh,x_i+\eta\dx,\hat{Y}^h_n)d\eta \\
-\frac 1 6 \,h  \dx^2\mu_X(\hat{Y}^h_n)\partial^3_xu(nh,x_{i},\hat{Y}^h_n).
\end{array}
\end{equation}

\noindent
\textbf{Step 3. Rearranging the terms.} By resuming, from \eqref{I1} and \eqref{I2} we have
$$
\begin{array}{l}
I_1-I_2\\
=h\partial_tu(nh,x_i,\hat{Y}^h_{n})+(\hat{Y}^h_{n+1}-\hat{Y}^h_n)\partial_yu(nh,x_i,\hat{Y}^h_{n})+h\mu_X(\hat{Y}^h_n)\partial_xu(nh,x_{i},\hat{Y}^h_n)\\
\quad+\frac 1 2 \big[(\hat{Y}^h_{n+1}-\hat{Y}^h_n)^2\partial^2_yu(nh,x_i,\hat{Y}^h_{n}) +  h\,\sigma_X^2(\hat{Y}^h_n) \partial^2_xu(nh,x_{i},\hat{Y}^h_n)\big]\\ 
\quad +h\int(u(t,x+\gamma_X(\hat{Y}^n_h)\zeta,\hat{Y}^h_n)-u(t,x,\hat{Y}^h_n))\nu(\zeta)d\zeta\\
\quad+ \sum_{i=1}^4R_{i}(n,h,x_i,\hat{Y}^h_{n},\hat{Y}^h_{n+1}) + S(n,h,\hat{Y}^h_n, \hat{Y}^h_{n+1}),
\end{array}
$$
in which we have used the change of variable giving
$$
\int(u(t,x+z,\hat{Y}^h_n)-u(t,x,\hat{Y}^h_n))\nu_{\hat{Y}^h_n}(z)dz
=\!\!\int(u(t,x+\gamma_X(\hat{Y}^n_h)\zeta,\hat{Y}^h_n)-u(t,x,\hat{Y}^h_n))\nu(\zeta)d\zeta
$$
and where
\begin{equation}\label{R4}
\begin{array}{l}
R_4(n,h,x_i,\hat{Y}^h_{n})
=
h	\sum_l \big[u(t,x_i+l\dx,\hat{Y}^h_{n})-u(t,x_{i},\hat{Y}^h_{n})\big]\nu_{\hat{Y}^h_n}(l\dx)\dx \\
-h\int\big[u(t,x_i+z,\hat{Y}^h_{n})-u(t,x_i,\hat{Y}^h_{n})\big]\nu_{\hat{Y}^h_n}(z)dz.
\end{array}
\end{equation} 
By passing to the conditional expectation and by using formulas \eqref{momento1-gen}, \eqref{momento2-gen} and \eqref{momento3-gen} for the local moments of order 1, 2 and 3, we obtain
$$
\begin{array}{l}
\widetilde{\mathcal{R}}_n^h(x_i,\hat{Y}^h_n):=\E[I_1-I_2\mid \hat{Y}^h_n]
=h(\partial_tu(nh,x_i,\hat{Y}^h_{n})+\L u(nh,x_i,\hat{Y}^h_{n}))\\
\quad +\sum_{i=1}^4\E[R_{i}(n,h,x_i,\hat{Y}^h_{n},\hat{Y}^h_{n+1})\mid \hat{Y}^h_n]+\E(S(n,h,x_i,\hat{Y}^h_n, \hat{Y}^h_{n+1})\mid \hat{Y}^h_n)\\
\quad 	= \sum_{i=1}^4\E[R_{i}(n,h,x_i,\hat{Y}^h_{n},\hat{Y}^h_{n+1})\mid \hat{Y}^h_n]+\sum_{i=1}^2 S_i(n,h,x_i,\hat{Y}^h_n).
\end{array}
$$
Here we have used the following facts: $u$ solves \eqref{PDEB}; $\E(S(n,h,x_i,\hat{Y}^h_n, \hat{Y}^h_{n+1})\mid \hat{Y}^h_n)=\sum_{i=1}^2 S_i(n,h,x_i,\hat{Y}^h_n)$, with (recall the definition of $S$ in  \eqref{S} and of the local moments $f_h$, $g_h$ and $j_h$  in \eqref{momento1-gen}, \eqref{momento2-gen} and  \eqref{momento3-gen})
\begin{equation}\label{S1}
\begin{array}{l}
S_1(n,h,x_i,\hat{Y}^h_{n}) \\
=f_h(\hat{Y}^h_n)\partial_yu(nh,x_i,\hat{Y}^h_n)+\frac 12g_h(\hat{Y}^h_n)\partial^2_yu(nh,x_i,\hat{Y}^h_n)+\frac 16j_h(\hat{Y}^h_n)\partial^3_yu(nh,x_i,\hat{Y}^h_n) \\ 
\quad +\partial_y\partial_tu(nh,x_i,\hat{Y}^h_{n})\,h(\mu_Y(\hat{Y}^h_{n})h+f_h(\hat{Y}^h_n)),
\end{array}
\end{equation} 
\begin{equation}\label{S2}
\begin{array}{l}
S_2(n,h,x_i,\hat{Y}^h_{n}) 
=h(h\mu_Y(\hat{Y}^h_n)+f_h(\hat{Y}^h_n))\times\\
\times\sum_l \nu_{\hat{Y}^h_n}(l\dx)\big[\partial_yu(nh,x_i+l\dx,\hat{Y}^h_{n})-\partial_yu(nh,x_{i},\hat{Y}^h_{n})\Big]\dx.
\end{array}
\end{equation} 
\noindent
\textbf{Step 4. Estimate of the remainder.} 
Hereafter, $C$ denotes a positive constant which may vary from a line to another and is independent of $n,h,\dx$.

By \eqref{mmm}, we have to study 
$\mathcal{R}_n^h(\cdot ,\hat{Y}^h_n)=(A^h_{\dx})^{-1}(\hat{Y}^h_n)\widetilde{\mathcal{R}}_n^h(\cdot ,\hat{Y}^h_n)$. By  Lemma  \ref{norma} it follows that  $\sup_{y\in\mathcal{D}}|(A^h_{\dx})^{-1}(y)|_2\leq 1$, so
$$
\begin{array}{l}
\E\big[e^{\sum_{l=1}^{n}2\lambda c_\nu h}|\mathcal{R}_n^h(\cdot,\hat{Y}^h_n)|^2_{2}\big]
\leq e^{2\lambda c_\nu T}\E\big[|\widetilde{\mathcal{R}}_n^h(\cdot,\hat{Y}^h_n)|^2_{2}\big]\\
\quad \leq C	\sum_{i=1}^4\E\big[|R_{i}(n,h,\cdot,\hat{Y}^h_{n},\hat{Y}^h_{n+1})|^2_{2}\big] +\sum_{i=1}^2\E\big[|S_{i}(n,h,\cdot,\hat{Y}^h_{n})|^2_{2}\big] .
\end{array}
$$
Hence it suffices to prove that the above 6 terms are all upper bounded by $Ch^2(h+\dx^2)^2$. The inequalities studied in $(ii)$ of  Lemma \ref{lemma-poisson} now come on.

Consider first $R_1$ in \eqref{R1} and in particular, the first addendum therein. Set
$$
g_n(x)= h^2\int_0^1(1-\tau)\partial^2_tu(nh+\tau h,x,\hat{Y}^h_{n+1})d\tau.
$$
Since $u\in C^{2,6}_{\pol,T}(\R,\mathcal{D})$,  $\partial^k_xg_n\in L^2(\R,dx)$ for every $k=0,1,2$ and $|\partial^k_xg_n|_{L^2}$ $\leq Ch^2(1+|\hat{Y}^h_{n+1}|^a)$. So, by using \eqref{poisson2},
$$
|g_n|_2^2\leq C h^4(1+|\hat{Y}^h_{n+1}|^a)^2.
$$
Similar estimates hold for the other terms in $R_1$, so we can write
$$
\begin{array}{rl}
|R_{1}(n,h,\cdot,\hat{Y}^h_{n},\hat{Y}^h_{n+1})|^2_{2}\leq
C\big[&\!\!\!\!h^4(1+|\hat{Y}_n^h|^{a})^2+|\hat{Y}_{n+1}-\hat{Y}_n|^{8}(1+|\hat{Y}_n^h|^{a}+|\hat{Y}_{n+1}^h|^{a})^2\\
&+h^2|\hat{Y}_{n+1}-\hat{Y}_n|^4(1+|\hat{Y}_n^h|^{a})^2\big].
\end{array}
$$
By using the increment estimates \eqref{stima_momento-gen}, the moment estimates \eqref{stima_incremento-gen} and the Cauchy-Schwartz inequality, we obtain
$$
\E\big[|R_{1}(n,h,\cdot,\hat{Y}^h_{n},\hat{Y}^h_{n+1})|^2_{2}\big]\leq Ch^4.
$$
The same arguments can be developed for $R_3$ in \eqref{R3} and $S_1$ in \eqref{S1}. These give
$$
\E\big[|R_3(n,h,\cdot,\hat{Y}^h_{n},\hat{Y}^h_{n+1})|^2_{2}\big]\leq C h^2\dx^4
\mbox{ and }\E\big[|S_1(n,h,\cdot,\hat{Y}^h_{n})|^2_{2}\big]	\leq Ch^4.
$$
In order to study $R_2$ in \eqref{R2}, consider the first term and set
$$
\begin{array}{l}
g_n(x)
=h^2	\sum_l \nu_{\hat{Y}^h_n}(l\dx)\dx
\times\\
\times	\int_0^1(1-\tau)\big[\partial_tu(nh+\tau h,x+l\Delta x,\hat{Y}^h_{n+1})-\partial_tu(nh+\tau h,x,\hat{Y}^h_{n+1})\big]d\tau.
\end{array}
$$
We notice that $g_n\in C^2$. By the Cauchy-Schwarz inequality for the (discrete) finite measure $\nu_{\hat{Y}^h_n}(l\dx)\dx$, $l\in\Z$, we have
$$
\begin{array}{l}
|\partial_x^kg_n(x)|^2
\leq Ch^4\sum_l \nu_{\hat{Y}^h_n}(l\dx)\dx
\times\\
\times	\int_0^1(1-\tau)^2\Big(\big|\partial_x^k\partial_tu(nh+\tau h,x+l\Delta x,\hat{Y}^h_{n+1})\big|^2+\big|\partial_x^k\partial_tu(nh+\tau h,x,\hat{Y}^h_{n+1})\big|^2\Big)d\tau.
\end{array}
$$
This gives $|\partial^k_x g_n|_{L^2}\leq C h^2(1+|\hat{Y}^h_{n+1}|^a)$ and, by \eqref{poisson2}, $|g_n|_2^2\leq C h^4 (1+|\hat{Y}^h_{n+1}|^a)^2$. By developing the same arguments to the other terms in $R_2$, we obtain 
$$
\begin{array}{l}
|R_{2}(n,h,\cdot,\hat{Y}^h_{n},\hat{Y}^h_{n+1})|^2_{2}\leq
C\big[h^4(1+(\hat{Y}_n^h)^{a})^2+
h^2|\hat{Y}_{n+1}-\hat{Y}_n|^{4}(1+|\hat{Y}_n^h|^{a})\big].
\end{array}
$$
And by passing to the expectation, we get $\E(|R_{2}(n,h,\cdot,\hat{Y}^h_{n},\hat{Y}^h_{n+1})|^2_{2})\leq Ch^4$. A similar approach can be used to handle $R_4$ in \eqref{R4} and in $S_2$ in \eqref{S2}, giving 
$$
\E\big[|R_4(n,h,\cdot ,\hat{Y}^h_{n})|_{2}^2\big]
\leq
Ch^2\dx^4 \mbox{ and }
\E\big[|S_2(n,h,\cdot,\hat{Y}^h_{n})|^2_{2}\big]\leq Ch^4.
$$
The proof is now completed.

\end{proof}%

	\subsubsection{Convergence in $l_\infty$-norm}\label{sect-linf}
	
	We consider here a different finite difference scheme for equation \eqref{PDE-barug-h}:
	%
	%
	we  still  approximate (explicit in time) the integral term $\L^{(y)}_{\mbox{\tiny{int}}}v$ in  \eqref{L_i} with a  trapezoidal rule, but  we use an  upwind first order scheme
	to approximate (implicit in time) the differential part $\L^{(y)}_{\mbox{\tiny{diff}}}v$ in  \eqref{L_d}.  As usually done in convection-diffusion problems, we distinguish the cases in which $\mu_X(y)$ is positive or negative in order to take into account the asymmetry given by the convection term and we use one sided difference in the appropriate direction.  
	Hence, the resulting scheme is 
	\begin{equation}\label{equaz2}
	A^h_{\dx}(y)v^n=B^h_{\dx}(y)v^{n+1},
	\end{equation}
	where $A^h_{\dx}(y)$ is the linear operator given by  
	\begin{equation}\label{A2}
	(A^h_{\dx})_{ij}(y)=\begin{cases}
	-\beta^h_{\dx}(y)-|\alpha^h_{\dx}(y)|\I_{\alpha^h_{\dx}(y)<0},\qquad &\mbox{ if }i=j+1,\\
	1+2\beta^h_{\dx}(y)+|\alpha^h_{\dx}(y)|,\qquad &\mbox{ if }i=j,\\
	-\beta^h_{\dx}(y)-|\alpha^h_{\dx}(y)|\I_{\alpha^h_{\dx}(y)>0},\qquad &\mbox{ if }i=j-1,\\
	0 &\mbox{ if }|i-j|>1
	\end{cases},
	\end{equation}  
	with
	$$
	\alpha^h_{\dx}(y)=\frac{h}{\dx}\mu_X(y), \qquad \beta^h_{\dx}(y)=\frac{h}{2\dx^2}\sigma^2_X(y),
	$$
	and $B^h_{\dx}(y)$ is the linear operator defined in \eqref{B}. Then we have:
	\begin{lemma}
		For every $y\in \D$, the operator $A^h_{\dx}(y):l_\infty(\mathcal{X})\rightarrow l_\infty(\mathcal{X})$ is  invertible and $$
		\sup_{y\in\mathcal{D}}|( A^h_{\dx})^{-1}(y)|_\infty\leq1
		$$
		And if $\frac{\nu'}{\nu},\frac{\nu''}{\nu}\in L^1(\R,d\nu)$ then  $\sup_{y\in\mathcal{D}}|B^h_{\dx}(y)|_\infty$ $\leq 1+2\lambda c_\nu$, $c_\nu$ being defined in \eqref{cnu}.
	\end{lemma}
	\begin{proof}
		We write $ A^h_{\dx}(y)=(1+\eta(y)) \Id-P(y)$, 
		where $\eta(y)=2\beta^h_{\dx}(y)+|\alpha^h_{\dx}(y)|\geq 0$ and $P_{ij}(y)=0$ if $|i-j|\neq 1$ and $P_{ij}=-(A^h_{\dx})_{ij}$ if $|i-j|=1$. It easily follows that $|P(y)|_\infty\leq \eta(y)$. Moreover, it is easy to see that the operator $A^h_{\dx}(y):l_\infty(\mathcal{X})\rightarrow l_\infty(\mathcal{X})	$ is invertible with inverse
		$$
		(A^h_{\dx})^{-1}(y)=(	(1+\eta(y)) \Id -P(y))^{-1}=\frac 1 {1+\eta(y)} \sum_{k=0}^\infty \frac {P(y)^k} {(1+\eta(y))^k }.
		$$
		This gives $|( A^h_{\dx})^{-1}(y)|_\infty\leq1$. The assertion for $B^h_{\dx}(y)$ follows from \eqref{B} and \eqref{cnu}.
	\end{proof}
	We can now state the convergence result in $l_\infty(\mathcal{X})$.

	\begin{theorem}
		\label{conv-H2}
		Let $u$ be defined in \eqref{european} and $(u^h_n)_{n=0,\ldots,N}$ be given by \eqref{backward-ter0} with the choice
		$$
		\Pi^h_{\dx}(y)=(A^h_{\dx})^{-1}B^h_{\dx}(y),
		$$
		$A^h_{\dx}(y)$ and $B^h_{\dx}(y)$ being given in \eqref{A2} and \eqref{B} respectively. Assume that:
		\begin{itemize}
			\item $\frac{\nu'}\nu,\frac{\nu''}\nu\in L^1(\R,d\nu)$;
			\item the Markov chain $(\hat{Y}^h_n)_{n=0,\dots, N}$ satisfies assumptions $\mathcal{H}_1$ and $\mathcal{H}_2$;
			\item $u\in  C^{\infty,4}_{\pol, T}( \R, \D)$.
		\end{itemize}
		Then, there exist $\bar h,C>0$ such that for every $h<\bar h$ and $\dx<1$  one has
		\begin{equation}
		| u(0,\cdot,Y_0)-u^h_{0}(\cdot,Y_0)|_{\infty} \leq C(h+\dx).
		\end{equation}
	\end{theorem}
	
	\begin{proof}
		The statement follows by applying Theorem \ref{convergencebates} once it is proved that $\mathcal{K}(\infty, 2\lambda c_\nu, h+\dx)$ holds. This is just a rewriting of the proof of   Theorem \ref{conv-H} in terms of the norm in $l_\infty(\mathcal{X})$. We only notice that, for handling the remaining terms, in $l_\infty$-norm we do not need to apply  \eqref{poisson2}, so we do not need more regularity for $u$. That's why the class $C^{\infty,4}_{\pol,T}(\R,\D)$ is enough. 
	\end{proof}	
	
\section{The hybrid procedure for the Heston or Bates model}\label{sect-bates}

As an application in finance, we consider the Heston \cite{heston} and the  Bates \cite{bates} model. In this framework,  $u(t,x,y)$ is in fact related to the value function at time $t$ of a European option with maturity $T$ and  (discounted) payoff $f$.

Recall that under the Heston or Bates model, the asset price process $S$ and the volatility process $Y$  evolve following the stochastic differential system
\begin{equation}\label{bates}
\begin{array}{ll}
&\displaystyle
\frac{dS_t}{S_{t^-}}= (r-\delta)dt+\mu \sqrt{Y_t}\, dZ^1_t+\gamma d\tilde H_t,   \\
&\displaystyle
dY_t= \kappa(\theta-Y_t)dt+\sigma\sqrt{Y_t}\,dZ^2_t,
\end{array}
\end{equation}
where  $S_0>0$, $Y_0> 0$, $Z=(Z^1,Z^2)$ is a correlated Brownian motions with $d\langle Z^1,Z^2\rangle_t$ $=\rho dt$, $|\rho|<1$, $\tilde H$ is a compound Poisson process with
intensity $\lambda$ and i.i.d. jumps $\{\tilde  J_k\}_k$ as in \eqref{H}. Here, $\gamma=1$ (Bates model) or $\gamma= 0$ (Heston model).
$r$ and $\delta$ are  the interest rate and the dividend interest rate respectively.	We assume, as usual, that the Poisson process $K$, the jump amplitudes $\{\tilde J_k\}_k$ and the correlated Brownian motion $(Z^1,Z^2)$ are independent. 

With a simple transformation, we can reduce the model \eqref{bates} to our reference model \eqref{generalsystem}. 
To get rid of the correlated Brownian motion,  we set $\bar\rho=\sqrt{1-\rho^2}$, $Z^2=W$ and $Z^1=\rho Z^2+\bar \rho B,$
$(B,W)$ denoting a standard $2$-dimensional Brownian motion. 
Moreover,	considering the process $X_t=\log S_t-\frac{\rho}{\sigma}Y_t$, the pair  $(X,Y)$ satisfies
\begin{equation}\label{SDE_bates}
\begin{array}{l}
dX_t= \mu_X(Y_t)dt+\bar\rho\,\sqrt{Y_t}\, dB_t+\gamma dH_t,\\
dY_t= \kappa(\theta-Y_t)dt+\sigma\sqrt{Y_t}\,dW_t,
\end{array}
\end{equation}
where 
$
\mu_X(y)=  r-\delta-\frac y2 -\frac \rho{\sigma}\kappa(\theta-y),
$
$H_t$ is the compound Poisson process written through the Poisson process $K$, with intensity $\lambda$, and the i.i.d. jumps $J_k=\log(1+\tilde J_k)$. 
The standard Bates model requires that $J_1$ has a normal law. But it is clear that the  convergence result holds for other laws such that the L\'evy measure $\nu$ satisfies the requests in  Theorem \ref{conv-H} or  Theorem \ref{conv-H2}. For example, these properties hold for the mixture of exponential laws used by Kou \cite{kou}.

We consider  the approximating Markov chain for the CIR process discussed in   Section \ref{sect-CIR} and the two possible finite difference operator discussed in  sections  \ref{sect-l2} and  \ref{sect-linf}. As an application of Theorem \ref{conv-H} and Theorem \ref{conv-H2}, we get the following convergence rate result of the hybrid method.

\begin{theorem}\label{Bates_conv}
	Let $(X,Y)$ be the solution to \eqref{SDE_bates} and let $(\hat{Y}^h_n)_{n=0,\dots, N}$ be the Markov chain introduced in    Section \ref{sect-CIR} for the approximation of the CIR process $Y$. Let $u(t,x,y)=\E(f(X_T^{t,x,y},Y_T^{t,y}))$ be as in \eqref{european} and $(u^h_n)_{n=0,\ldots,N}$ be given by \eqref{backward-ter0} with the choice
	$$
	\Pi^h_{\dx}(y)=(A^h_{\dx})^{-1}B^h_{\dx}(y).
	$$
	
	\begin{itemize}
		\item[$(i)$] 
		$\mathrm{[Convergence\  in\ }l_2(\mathcal{X})]$ Suppose that 
		\begin{itemize}
			\item [$\bullet$]
			$A^h_{\dx}(y)$ and $B^h_{\dx}(y)$ are defined in \eqref{A} and \eqref{B} respectively;
			\item [$\bullet$]
			$\frac{\nu'}{\nu},\frac{\nu''}{\nu}\in L^2(\R,d\nu)$ and $\nu$ has finite moments of any order;
			\item [$\bullet$]
			$\partial^{2j}_xf\in C^{2,6-j}_{\pol}(\R, \R_+)$ for every $j=0,\ldots,6$.
		\end{itemize} 
		Then, there exist $\bar h,C>0$ such that for every $h<\bar h$ and $\dx<1$ one has
		$$
		| u(0,\cdot,Y_0)-u^h_{0}(\cdot,Y_0)|_{2} \leq CT(h+ \dx^2).
		$$
		
		\item[$(ii)$] 
		$\mathrm{[Convergence\  in\ } l_\infty(\mathcal{X})]$ Suppose that 
		\begin{itemize}
			\item [$\bullet$]
			$A^h_{\dx}(y)$ and $B^h_{\dx}(y)$ are defined in \eqref{A2} and \eqref{B} respectively;
			\item [$\bullet$]
			$\frac{\nu'}{\nu},\frac{\nu''}{\nu}\in L^1(\R,d\nu)$  and $\nu$ has finite moments of any order;
			\item [$\bullet$]
			$\partial^{2j}_xf\in C^{\infty,4-j}_{\pol}(\R, \R_+)$ for every $j=0,\ldots,4$.
		\end{itemize} 
		Then, there exist $\bar h,C>0$ such that for every $h<\bar h$ and $\dx<1$ one has
		$$
		| u(0,\cdot,Y_0)-u^h_{0}(\cdot,Y_0)|_{\infty} \leq CT(h+ \dx).
		$$
	\end{itemize} 
	
\end{theorem}
\begin{proof}
	We apply   Theorem \ref{conv-H} for $(i)$ and Theorem  \ref{conv-H2} for $(ii)$. Following Theorem \ref{conv_T}, the assumptions $\mathcal{H}_1$ and $\mathcal{H}_2$ hold (see also Proposition \ref{moments-2}). So, we need only to prove that if $\partial_x^{2j}f\in C^{2,6-j}_{\pol}(\R, \R_+)$ as $j=0,1,\ldots,6$, resp. $\partial_x^{2j}f\in C^{\infty,4-j}_{\pol}(\R, \R_+)$ as $j=0,1,\ldots,4$, then $u\in C^{2,6}_{\pol,T}(\R, \R_+)$, resp. $u\in C^{\infty,4}_{\pol,T}(\R, \R_+)$. This is proved in next  Proposition  \ref{prop-reg-new} (set $\rho=0$, $\mathfrak{a}=r-\delta-\frac \rho{\sigma}\kappa\theta$ and $\mathfrak{b}=\frac \rho{\sigma}\kappa-\frac 12$ therein),  the whole Section  \ref{appendix-reg} being devoted to.
\end{proof}

\begin{remark}
	Another  example of interest in finance is the Bates-Hull-White mo\-del \cite{bctz}, which is a Bates model coupled with a stochastic interest rate. The dynamics follows  \eqref{bates} in which $r$ is not constant but given by the Vasicek model 
	$$
	dr_t=\kappa_r(\theta_r-r_t)dt+\sigma_rdZ^3_t,
	$$
	$Z^3$ being a Brownian motion correlated with $Z^1$ (and possibly $Z^2$). Here, there is no global transformation allowing one to reduce to our reference model. Nevertheless, a similar convergence result can be proved by means of the local transformation introduced in \cite{bctz} (Section  4.1), acting on each time interval $[nh,(n+1)h]$.	
\end{remark}

%

\subsection{A regularity result for the Heston PDE/Bates PIDE}\label{appendix-reg}

We deal here with a slightly more general model: we consider the SDE 
\begin{equation}\label{SDE}
\begin{array}{l}
dX_t= \left(\mathfrak{a}+\mathfrak{b}Y_t\right)dt+\sqrt{Y_t}\, dW^1_t+\gamma_XdH_t, \\
dY_t= \kappa(\theta-Y_t)dt+\sigma\sqrt{Y_t}\,dW^2_t,
\end{array}
\end{equation}
where $W^1,W^2$ are correlated Brownian motions  with $d\langle W^1,W^2\rangle_t=\rho dt$ and $H$ is a compound Poisson process with intensity  $\lambda$ and  L\'evy measure $\nu$, which is assumed hereafter to have finite moments of any order. Here, $\mathfrak{a},\mathfrak{b}\in\R$  and $\gamma_X\in\{0,1\}$ denote constant parameters. Note that 
when $\mathfrak{a}=r-\delta$, $\mathfrak{b}=-\frac 12$ and $\gamma_X=0$ (resp. $\gamma_X=1$), then $(X,Y)$ is the standard Heston (resp. Bates) model for the log-price and volatility. When instead $\rho=0$, $\mathfrak{a}=r-\delta-\frac \rho{\sigma}\kappa\theta$ and $\mathfrak{b}=\frac \rho{\sigma}\kappa-\frac 12$, we recover the equation \eqref{SDE_bates} discussed in   Theorem \ref{Bates_conv}.

Let  $\L$ denote the infinitesimal generator associated to \eqref{SDE}, that is,
\begin{equation}	\label{app-L}
\L u=\frac y 2 \left( \partial^2_x u +2\rho\sigma \partial_x\partial_yu+\sigma^2\partial^2_yu  \right)
+  \left(\mathfrak{a}+\mathfrak{b}y \right)\partial_xu
+\kappa(\theta-y)\partial_yu+\mathcal{L}_{\mbox{{\tiny int}}} u,
\end{equation}
where, hereafter, we set
$
\mathcal{L}_{\mbox{{\tiny int}}} u(t,x,y)=\int \big[u(t,x+\gamma_X\zeta,y)-u(t,x,y)\big]\nu(\zeta)d\zeta.
$

So, the present section is devoted to the proof of the following result.
\begin{proposition} \label{prop-reg-new}
	Let $p\in [1,	\infty]$, $q\in\N$ and  suppose that 
	$\partial_x ^{2j}f\in C^{p,q-j}_{\pol}(\R,\R_+)$ for every $j=0,1,\ldots,q$. 
	Set
	$$
	u(t,x,y)=\E\big[f(X^{t,x,y}_T,Y^{t,y}_T)\big].
	$$
	Then $u\in C^{p,q}_{\pol, T}(\R,\R_+)$. 
	Moreover, the following stochastic representation holds: for $m+2n\leq 2q$,
	\begin{equation}\label{stoc_repr-new}
	\begin{array}{l}
	\partial^{m}_x\partial^{n}_yu(t,x,y)=\E\left[e^{-n\kappa (T-t)} \partial^m_x\partial^n_yf(X^{n,t,x,y}_T,Y^{n,t,x,y}_T)\right]\\
	\quad+
	n\,\E\left[\int_t^T\left[\frac 12 \partial^{m+2}_x\partial^{n-1}_yu+\mathfrak{b}\partial^{m+1}_x\partial^{n-1}_yu\right](s,X^{n,t,x,y}_s,Y^{n,t,x,y}_s)ds  \right],
	\end{array}
	\end{equation}
	where $\partial^{m}_x\partial^{n-1}_y u:=0$ when $n=0$ and $(X^{n,t,x,y},Y^{n,t,x,y})$, $n\geq 0$, denotes the solution starting from $(x,y)$ at time $t$ to the SDE \eqref{SDE}  with parameters
	\begin{equation}\label{parameters-new}
	\rho_n=\rho,\quad \mathfrak{a}_n=\mathfrak{a}+n\rho\sigma,\quad \mathfrak{b}_n=\mathfrak{b},\quad \kappa_n=\kappa,\quad \theta_n=\theta+\frac{n\sigma^2}{2\kappa},\quad \sigma_n=\sigma.
	\end{equation}
	In particular, if $q\geq 2$ then $u\in C^{1,2}([0,T]\times \bar{\mathcal{O}})$, $\bar{\mathcal{O}}=\R\times\R_+$, solves the PIDE
	\begin{equation}\label{PIDE-new}
	\left\{
	\begin{array}{ll}
	\partial_t u(t,x,y)+ \L u(t,x,y)= 0, &(t,x,y)\in [0,T)\times\bar{\mathcal{O}},\\
	u(T,x,y)= f(x,y), & (x,y)\in \bar{\mathcal{O}}.
	\end{array}
	\right.	
	\end{equation}
\end{proposition}

\begin{remark}
	For our purposes, we need both the polynomial growth condition for $(x,y)\mapsto u(t,x,y)$ and the $L^p$	property for $x\mapsto u(t,x,y)$, and similarly for the derivatives. A closer look to the proof of   Proposition \ref{prop-reg-new} shows that the result holds also when one is not interested in the latter $L^p$ condition. In this case,   Proposition \ref{prop-reg-new} reads:  for $q\in\N$, if $\partial_x ^{2j}f\in C^{q-j}_{\pol}(\R\times\R_+)$ for every $j=0,1,\ldots,q$ then
	$u\in C^q_{\pol,T}(\R\times \R_+)$. Moreover, the stochastic representation \eqref{stoc_repr-new} holds and, if $q\geq 2$, $u$ solves PIDE \eqref{PIDE-new}.
\end{remark}	

As an immediate consequence of  Proposition \ref{prop-reg-new}, we obtain the already known regularity result for the CIR process which has been already proved in  Proposition 4.1 of  \cite{A-MC}.

\begin{corollary}\label{corollary-cir}
	Assume that $f=f(y)$ and set
	$
	u(t,y)=\E\big[f(Y^{t,y}_T)\big].
	$
	If $f\in C^q_\pol(\R_+)$, then $u\in C^q_{\pol, T}(\R_+)$. 
	Moreover, for $n\leq q$,
	$$
	\partial^n_y u(t,y)=\E\left[e^{-n\kappa (T-t)} 	\partial^n_yf(Y^{n,t,y}_T)\right],
	$$
	where $Y^{n,t,y}$ denotes a CIR process starting from $y$ at time $t$ which solves the CIR dynamics with parameters $\kappa_{n}=\kappa$, $\theta_{n}=\theta+\frac {n\sigma^2} {2\kappa}$, $\sigma_n=\sigma$. 
	In particular, if  $q\geq 2$ then $u\in C^2_{\pol,T}(\R_+)$ solves the PDE
	$$
	\left\{
	\begin{array}{ll}	
	\partial_tu+ \mathcal{A} u= 0,& (t,y)\in [0,T)\times \R_+,\\
	u_n(T,y)= \partial^n_yf(y),& y\in \R_+,
	\end{array}
	\right.
	$$	
	where $\mathcal{A}$ is the CIR infinitesimal generator given in \eqref{A-op}.
\end{corollary}

We first need some preliminary results. First of all, recall that $X$ and $Y$ have uniformly bounded moments:
for every $T>0$ and $a\geq 1$ there exist $A>0$  such that for every $t\in[0,T]$,
\begin{equation}\label{moments_XY}
\sup_{s\in[t,T]} \E[|X^{t,x,y}_s|^a]\leq A(1+|x|^{a}+y^{a})\mbox{ and }
\sup_{s\in[t,T]} \E[|Y^{t,y}_s|^a]\leq A(1+y^a).
\end{equation}
For the second property in  \eqref{moments_XY}, we refer, for example, to \cite{A-MC}, whereas the first one follows from standard techniques.
\begin{lemma}\label{lemma_reg}
	Let $p\in[0,\infty]$, $ g\in C^{p,0}_{\pol}(\R,\R_+)$,
	$ h\in C^{p,0}_{\pol,T}(\R,\R_+)$  and consider the function
	\begin{equation}\label {u-g-f}
	u(t,x,y)=\E\Big[e^{\varrho (T-t)} g(X^{t,x,y}_T,Y^{t,y}_T)-\int_t^Te^{\varrho (s-t)} h(s,X^{t,x,y}_s,Y^{t,y}_s)ds  \Big],
	\end{equation}
	where $\varrho\in\R$. Then $u\in C^{p,0}_{\pol,T}(\R,\R_+)$. 
\end{lemma}
\begin{proof}
	We set
	$$
	u_1(t,x,y)=\E\Big[e^{\varrho (T-t)} g(X^{t,x,y}_T,Y^{t,y}_T)\Big], 
	u_2(t,x,y)=\E\Big[\int_t^Te^{\varrho (s-t)} h(s,X^{t,x,y}_s,Y^{t,y}_s)ds  \Big]
	$$
	%
	and we show that, for $i=1,2$, $u_i\in C^{p,0}_{\pol,T}(\R,\R_+)$. We prove it for $i=2$, the case $i=1$ being similar and easier.

	Fix $(t,x,y)\in [0,T]\times \R\times \R_+$ and let $(t_n,x_n,y_n)_n\subset [0,T]\times\R\times\R_+$ be such that $(t_n,x_n,y_n)\to (t,x,y)$ as $n\to\infty$. One can easily prove that, for every fixed $s\geq t_n\vee t$,  $(X_s^{t_n,x_n,y_n},Y_s^{t_n,y_n})\rightarrow(X_s^{t,x,y},Y_s^{t,y})$ in probability. 
	%
	We write $u_2$ as 
	$$
	u_2(t,x,y)=\int_0^T\I_{s>t}e^{\varrho (s-t)}\E\left[ h(s,X^{t,x,y}_s,Y^{t,y}_s) \right]ds.
	$$
	Since $ h$ is continuous,	for $s>t_n\vee t$ the sequence $( h(s,X^{t_n,x_n,y_n}_s,Y^{t_n,y_n}_s))_n$ converges in probability to $ h(s,X^{t,x,y}_s,Y^{t,y}_s)$.  By the polynomial growth of $h$ and \eqref{moments_XY}, for $p>1$ we have
	\begin{align}\label{a}
	\sup_{n} \E[	| h(X_s^{t_n,x_n,y_n},Y_s^{t_n,y_n})|^p]&\leq C\sup_n\E[1+|X_s^{t_n,y_n}|^{ap}+(Y_s^{t_n,y_n})^{ap}]<\infty.
	\end{align}
	Thus, $( h(X_s^{t_n,x_n,y_n},Y_s^{t_n,y_n}))_n$ is uniformly integrable, so $ \{h(X_s^{t_n,x_n,y_n},Y_s^{t_n,y_n})\}_n$ converges to $h(X_s^{t,x,y},Y_s^{t,y})$ in $L^1(\Omega)$ and
	$$
	\I_{s>t_n}\E\left[e^{\varrho (s-t_n)} h(s,X^{t_n,x_n,y_n}_s,Y^{t_n,y_n}_s) \right]\to \I_{s>t}\E\left[e^{\varrho (s-t)} h(s,X^{t,x,y}_s,Y^{t,y}_s) \right],
	$$
	a.e. $s\in [0,T] $.
	By \eqref{a}, $u_2(t_n,x_n,y_n)\to u_2(t,x,y)$ thanks to the Lebesgue's dominated convergence and moreover, $u_2$ grows polynomially. So, $u_2\in \mathcal{C}_{\pol,T}(\R\times\R_+)$.
	Fix now $p\neq\infty$. We notice that $X^{t,x,y}_s=x+Z^{t,y}_s$, so
	$$
	\begin{array}{l}
	\sup_{t\leq T}\|u_2(t,\cdot,y)\|_{L^p(\R,dx)}
	= \sup_{t\leq T}\| \E [\int_t^Te^{\varrho (s-t)} h(s,X^{t,\cdot,y}_s,Y^{t,y}_s)ds   ] \|_{L^p(\R,dx)}\\
	\quad\leq C\sup_{t\leq T}\E [ \int_t^T \| h(s,X^{t,\cdot,y}_s,Y^{t,y}_s)\|^p_{L^p(\R,dx)} ]^{1/p}\\
	\quad = C\sup_{t\leq T}\E[ \int_t^T \| h(s,\cdot +Z^{t,y}_s,Y^{t,y}_s)\|^p_{L^p(\R,dx)} ]^{1/p}\\
	\quad = C\sup_{t\leq T}\E [ \int_t^T \| h(s,\cdot ,Y^{t,y}_s)\|^p_{L^p(\R,dx)}]^{1/p}
	\leq C T\sup_{t\leq s\leq T}(1+\E[(Y_s^{t,y})^{pa}])^{1/p}
	\end{array}
	$$
	in which we have used twice the Cauchy-Schwarz inequality. Then, by \eqref{moments_XY}, $u_2\in C_{\pol,T}^{p,0}(\R,\R_+)$. The case $p=\infty$ follows the same lines.
\end{proof}
To simplify the notation, from now on we set $\E^{t,x,y}[\cdot]=\E[\cdot|X_t=x,Y_t=y]$ and $\mathcal{O}=\R\times(0,\infty)$.
\begin{lemma}\label{u-pde}
	Let $ g\in \mathcal{C}_{\pol}(\bar{\mathcal{O}})$ and $ h\in C_{\pol,T}(\bar{\mathcal{O}})$ be such that $\mathcal{O}\ni z\mapsto  h(t,z)$ is locally H\"older continuous uniformly on the compact sets of $[0,T)$. 
	Let $u$ be defined in \eqref {u-g-f}. Then, $u\in\mathcal{C}([0,T]\times\bar{\mathcal{O}})\cap\mathcal{C}^{1,2}([0,T)\times \mathcal{O})$ and solves the PIDE
	\begin{equation}\label{PDE-u}
	\left\{
	\begin{array}{ll}
	\partial_t u+ \L u+\varrho u= h,&\mbox{ in }[0,T)\times\mathcal{O},\\
	u(T,z)= g(z),&\mbox{ in }\mathcal{O}.
	\end{array}\right.
	\end{equation}
	Moreover, if the Feller condition $2\kappa\theta\geq \sigma^2$ holds then $u$ is the unique solution to \eqref{PDE-u} in the class $C_{\pol,T}(\bar{\mathcal{O}})$.
	
\end{lemma}
The proof employs standard techniques, see e.g. Proposition 3.2 in \cite{ET} with the use of  classical results in parabolic PIDEs theory from \cite{GM, MP}. The uniqueness of the solution under the Feller condition follows from the fact that the CIR process never hits $0$. So, we omit this proof.

\begin{lemma}
	\label{lemma-reg}
	
	Let $u$ be defined in \eqref{u-g-f}, with $g$ and $h$ such that, as $j=0,1$,
	$\partial_x^{2j}g\in C^{1-j}_{\pol}(\bar{\mathcal{O}})$  and  $\partial_x^{2j}h\in \mathcal{C}^{1-j}_{\pol,T}(\bar{\mathcal{O}})$.
	Then 
	$u\in \mathcal{C}^1_{\pol,T}(\bar{\mathcal{O}})$. Moreover, $\partial^2_xu\in \mathcal{C}_{\pol,T}(\bar{\mathcal{O}})$ and one has
	\begin{equation}
	\partial^m_x u(t,x,y)
	=\E^{t,x,y}\Big[e^{\varrho (T-t)} \partial^m_x g(X_T,Y_T)-\int_t^Te^{\varrho (s-t)} \partial^m_x h(s,X_s,Y_s)ds  \Big],\  m=1,2,\label{u-x-xx}
	\end{equation}
	\begin{equation}\label{u-y}
	\begin{array}{rl}
	\partial_y u(t,x,y)
	=\E^{t,x,y}\Big[&e^{(\varrho-\kappa) (T-t)} \partial_yg(X^*_T,Y^*_T)\\
	&+\int_t^Te^{(\varrho-\kappa) (T-s)} [\partial_yh+\frac 12 \partial^2_xu+\mathfrak{b}\partial_xu](s,X^*_s,Y^*_s)ds  \Big],
	\end{array}
	\end{equation}
	where
	$(X^*_t,Y^*_t)$ solves  \eqref{SDE} with new parameters 	$\rho_*=\rho$, $\mathfrak{a}_*=\mathfrak{a}+\rho\sigma$, $\mathfrak{b}_*=\mathfrak{b}$,  $\kappa_*=\kappa$, $\theta_*=\theta+\frac{\sigma^2}{2\kappa}$, $\sigma_*=\sigma$.
\end{lemma}

\begin{proof}
	First, the stochastic flow w.r.t. $x$ is differentiable (here, $(X^*)^{t,x,y}_s=x+Z^{ t,y}_s$ and $Z^{ t,y}_s$ does not depend on $x$). Hence, by using the polynomial growth hypothesis, by  \eqref{u-g-f} one gets  \eqref{u-x-xx}. Let us prove \eqref{u-y}.
	
	By   Lemma \ref{u-pde} $u$  solves \eqref{PDE-u}. So, setting $v=\partial_y u$, by derivating \eqref{PDE-u} one has
	$$
	\left\{
	\begin{array}{ll}
	\partial_t v+ \L_* v+\varrho_* v= h_*,&\mbox{ in }[0,T)\times\mathcal{O},\\
	v(T,z)= g_*(z),&\mbox{ in }\mathcal{O}.
	\end{array}	
	\right.
	$$
	where $\L_*$ is the infinitesimal generator of $(X^*,Y^*)$ and 
	$\varrho_*=\varrho-\kappa$, $h_*=\partial_y h-\mathfrak{b}\partial_xu-\frac 12\partial^2_xu$, $g_*=\partial_yg$.
	By using \eqref{u-x-xx} and    Lemma \ref{lemma_reg}, $h_*\in C_{\pol, T}(\bar{\mathcal{O}})$. Moreover, the Feller condition $2\kappa_*\theta_*\geq \sigma^2_*$ holds, and by  Lemma \ref{u-pde} the unique solution with polynomial growth in $(x,y)$ to the above PIDE is
	$$
	\bar v(t,x,y)
	=\E^{t,x,y}\Big[e^{\varrho (T-t)} g_*(X^*_T,Y^*_T)-\int_t^Te^{\varrho (s-t)} h_*(s,X^*_s,Y^*_s)ds \Big].
	$$
	In order to identify $\bar v$ with $v=\partial_yu$, one should know that $\partial_yu\in C_{\pol, T}(\mathcal{O})$.  If the diffusion coefficient of $Y$ was more regular, one could use arguments from the stochastic flow. But this is not the case, hence we  use a density argument inspired by \cite{ET}. 
	
	For $k\geq 1$, let $\varphi_k$  be a $C^\infty(\R)$ approximation of $\sqrt{|y|}$ such that $\varphi_k(y)\geq 1/k$, $\varphi_k(y)\to \sqrt{|y|}$ uniformly on the compact sets of $[0,+\infty)$ and $\varphi^2_k$ is Lipschitz continuous uniformly in $k$ (which means that $\varphi_k\varphi'_k$ is bounded uniformly in $k$).  Consider the diffusion process $(X^k,Y^k) $ defined by
	\begin{equation}\label{SDE-n}
	\left\{
	\begin{array}{l}
	dX^k_t=\left( \mathfrak{a}+\mathfrak{b}Y^k_t\right)dt +\varphi_k(Y^k_t)dB_t+dH_t,\\
	dY^k_t=\kappa(\theta-Y^k_t)dt+\sigma\varphi_k(Y^k_t)dW_t, 
	\end{array}
	\right.
	\end{equation}
	whose generator is
	$$
	\L_ku= \frac {\varphi^2_k(y)} 2 \left(\partial^2_x u+ 2\rho\sigma\partial_x\partial_y u  
	+ \sigma^2 \partial^2_y u\right)
	+\left(\mathfrak{a}+\mathfrak{b} y \right)\partial_x u+\kappa(\theta-y)\partial_y u  + \II u.
	$$
	Set
	$$
	u^k(t,x,y)=\E^{t,x,y}
	\Big[e^{\varrho (T-t)} g(X^k_T,Y^k_T)-\int_t^Te^{\varrho (s-t)}h(s, X^k_s, Y^k_s)ds\Big].
	$$
	Le us first show that $\partial_yu^k\in C_{\pol,T}(\mathcal{O})$. Since the diffusion coefficients associated to $(X^k,Y^k)$ are good enough, we can consider the first variation process: by calling $Z^{k,t,x,y}_s=(\partial_yX^{k,t,x,y}_s, \partial_y Y^{k,t,x,y}_s)$, we get
	$$
	\begin{array}{rl}
	\partial_y  u^k(t,x,y)
	=&\E[e^{\varrho (T-t)}\langle \nabla_{x,y}g(X^{k,t,x,y}_T,Y^{k,t,x,y}_T),Z^{k,t,x,y}_T\rangle]\\
	&-\int_t^Te^{\varrho (s-t)}
	\E[\langle\nabla_{x,y}h(s,X^{k,t,x,y}_s,Y^{k,t,x,y}_s), Z^{k,t,x,y}_s\rangle]ds.
	\end{array}
	$$
	The functions $g,h$ and their derivatives have polynomial growth, so
	$$
	\begin{array}{rl}
	\left|\partial_y u^k(t,x,y)\right|
	\leq &\E[C(1+|X^{k,t,x,y}_T|^a+|Y^{k,t,x,y}_T|^a)|Z^{k,t,x,y}_T|]\\
	&+\int_t^Te^{\varrho (s-t)}
	\E[C(1+|X^{k,t,x,y}_s|^a+|Y^{k,t,x,y}_s|^a)|Z^{k,t,x,y}_s|]ds
	\end{array}
	$$
	and the usual $L^p$-estimates give
	$$
	\sup_{t<T}\left|\partial_y  u^k(t,x,y)\right|\leq C_k(1+|x|^{a_k}+y^{a_k}),
	$$
	for suitable constants $C_k,a_k>0$.
	Moreover, from the standard theory of parabolic PIDEs, $u^k$ is a solution to
	$$
	\left\{
	\begin{array}{ll}
	\partial_t  u^k+ \L_k u^k+\varrho u^k= h,&\mbox{ in }[0,T)\times\mathcal{O},\\
	u^k(T,z)= g(z),&\mbox{ in }\mathcal{O}.
	\end{array}	
	\right.
	$$
	By differentiating, $v^k=\partial_y  u^k$ solves the problem
	$$
	\left\{
	\begin{array}{ll}
	\partial_t  v^k+ \L_{k,*} v^k+\varrho_* v^k= h_{k,*},&\mbox{ in }[0,T)\times\mathcal{O},\\
	v^k(T,z)= g_{*}(z),&\mbox{ in }\mathcal{O},
	\end{array}
	\right.	
	$$
	where
	$$
	\begin{array}{rl}
	\L_{k,*}v=
	& \frac {\varphi^2_k(y)} 2 \left(\partial^2_xv+ 2\rho\sigma\partial_x\partial_y v
	+ \sigma^2 \partial^2_yv \right)\\
	&+ \left(\mathfrak{a}+\mathfrak{b}y+2\rho\sigma\varphi_k\varphi'_k(y)\right)\partial_xv
	+\left(\kappa(\theta-y)+\sigma^2\varphi_k\varphi'_k(y)\right)\partial_yv+\II v 
	\end{array}
	$$
	and 
	$
	h_{k,*}=\partial_yh-\mathfrak{b} \partial_x u^k-\varphi_k\varphi'_k(y)\partial^2_x u^k.
	$
	By developing the same arguments as before, we get $h_{k,*}\in C_ {\pol, T}(\bar{\mathcal{O}})$.
	The PIDE for $v^k$ has a unique solution in $C_{\pol,T}(\mathcal{O})$ (recall that, by construction, the second order operator is uniformly elliptic). Thus, the Feynman-Kac formula gives
	$$
	\partial_yu^k(t,x,Y)
	=\E^{t,x,y}\Big[e^{\varrho (T-t)} g_*(X^{k,*}_T,Y^{k,*}_T)-\int_t^Te^{\varrho (s-t)} h_{k,*}(s,X^{k,*}_s,Y^{k,*}_s)ds  \Big],
	$$
	where $(X^{k,*},Y^{k,*})$ is the diffusion with infinitesimal generator given by $\L_{k,*}$.
	Now, the standard $L^p$ estimates for $(X^k,Y^k)$ and $(X^{k,*},Y^{k,*})$ hold uniformly in $k$ (recall that $\varphi_k$ is sublinear uniformly in $k$ and $\varphi_k\varphi'_k$ is bounded uniformly in $k$): for every $p\geq 1$ there exist $C,a>0$ such that
	$$
	\sup_k\sup_{t\leq T}\E^{t,x,y}\left( |X^k_t|^p+|Y^k_t|^p\right)
	+\sup_k\sup_{t\leq T}\E^{t,x,y}\left( |X^{k,*}_t|^p+|Y^{k,*}_t|^p\right)\leq C(1+|x|^a+|y|^a).
	$$
	This gives that
	$\sup_k\sup_{t<T}|u^k(t,x,y)|+\sup_k\sup_{t<T}\left|\partial_y u^k(t,x,y)\right|\leq C(1+|x|^a+|y|^a)$,
	for suitable $C,a>0$ (possibly different from the ones above). Moreover, the stability results in \cite{BMO} give
	$\lim_{k\to\infty}u^k(t,x,y)=u(t,x,y)$ and $\lim_{n\to\infty}\partial_y u^k(t,x,y)$ $=v(t,x,y)$
	for every $(t,x,y)\in [0,T)\times \mathcal{O}$. And thanks to the above uniform polynomial bounds for $u^k$ and $\partial_yu^k$, for every $\phi\in C^\infty (\mathcal{O})$ with compact support we easily get
	$$
	\begin{array}{l}
	\int v(t,x,y) \phi(x,y)dxdy
	=\int \lim_k \partial_y  u^k(t,x,y)\phi(x,y)dxdy\\
	=-\int \lim_k u^k(t,x,y)\partial_y\phi(x,y)dxdy
	=-\int u(t,x,y)\partial\phi(x,y)dxdy.
	\end{array}
	$$
	Therefore, $v(t,x,y)=\partial_y u(t,x,y)$ in $[0,T)\times \mathcal{O}$. 
\end{proof}

We can now prove the result which this section is devoted to.

\begin{proof}[Proof of   Proposition \ref{prop-reg-new}]
	We follow an induction on $q$. If $q=0$, 	  Lemma \ref{lemma_reg} gives the result. Suppose the statement is true up to $q-1\geq 0$ and let us prove it for $q$. 
	
	Take $f$ such that	$\partial_x ^{2j}f\in C^{p,q-j}_{\pol}(\R,\R_+)$ for every $j=0,1,\ldots,q$. Then, by induction, $\partial^{l}_t\partial^m_x\partial^n_y u\in C^{p,0}_{\pol,T}(\R,\R_+)$ when $2l+m+n\leq q-1$. So, we just need to prove that $\partial^l_t\partial^m_x\partial^n_y u\in C^{p,0}_{\pol, T}(\R,\R_+)$ for any $l,m,n$ such that $2l+m+n=q$. 
	
	Assume first $l=0$. For $n=0$, we use that $X_T^{t,x,y}=x+Z_T^{t,y}$ and we get
	$
	\partial^{m}_xu(t,x,y)=\E^{t,x,y}\big[\partial^m_xf(X_T,Y_T)\big].
	$
	Since $\partial^m_xf\in C^{p,0}_{\pol}(\R,\R_+)$ for any $m\leq 2q$,  by   Lemma \ref{lemma_reg} we obtain  $\partial^{m}_xu\in C^{p,0}_{\pol,T}(\R,\R_+)$	for every $m\leq 2q$.  
	
	Fix now $n>0$ and $m\geq 0$.   Recursively applying   Lemma \ref{lemma-reg}, we get  formula \eqref{stoc_repr-new}. Let us stress that,
	because of the presence of the derivatives $\partial^{m+2}_x\partial^{n-1}_yu$ and $\partial^{m+1}_x\partial^{n-1}_yu$ in \eqref{stoc_repr-new}, the recursively application of  Lemma \ref{lemma-reg} gives the constraint $m+2n\leq q$. 
	Then, by   Lemma \ref{lemma_reg}, it follows that $\partial^{m}_x\partial^{n}_yu \in C^{p,0}_{\pol, T}(\R,\R_+)$	for every $m, n\in \N$ such that $m+2n\leq 2q$, and in particular when $m+n=q$.
	
	Consider now the case $l>0$. By \eqref{stoc_repr-new},   Lemma \ref{u-pde} ensures that if $m+2n\leq 2q$ then $u_{n,m}=\partial^{m}_x\partial^{n}_y u$ solves
	\begin{equation}\label{pdeit-new}
	\left\{
	\begin{array}{ll}
	\partial_t u_{m,n}+ \L_{n}u_{m,n}-n\kappa u_{m,n}=-n \big[\frac 12 u_{m+2,n-1}+ \mathfrak{b}u_{m+1,n-1}\big], & \mbox{in }[0,T)\times\mathcal{O},\\
	u_{m,n}(T,x,y)=\partial^m_x\partial^n_y f(x,y), & \mbox{in }\mathcal{O},
	\end{array}
	\right.
	\end{equation}
	where $\L_n$ is the generator in \eqref{app-L} with the (new) parameters in \eqref{parameters-new}.
	Therefore, the general case concerning  $\partial^l_t\partial^m_x\partial^n_y u$ with $2l+m+n=q$ follows by an iteration on $l$: by \eqref{pdeit-new},
	$$
	\begin{array}{l}
	\partial^{l}_t\partial^{m}_x\partial^n_y u\\
	=-\L_n\partial^{l-1}_t\partial^{m}_x\partial^n_y u
	+n\kappa \partial^{l-1}_t\partial^{m}_x\partial^n_y u
	-n[\frac 12\partial^{l-1}_t\partial^{m+2}_x\partial^{n-1}_y u+\mathfrak{b}\partial^{l-1}_t\partial^{m+1}_x\partial^{n-1}_y u].
	\end{array}
	$$
\end{proof}

	\appendix
	
	\section{Lattice properties of the CIR approximating tree}\label{app-jumps}
	The aim of this section is to prove  Propostition \ref{propjumps}. 
	For later use, let us first give some (trivial) properties of the lattice. First, by construction, $k_d(n,k)\leq k<k_u(n,k)$, so that
	$
	y^{n+1}_{k_d(n,k)}\leq y^{n+1}_k\leq y^n_k\leq y^{n+1}_{k+1}\leq y^{n+1}_{k_u(n,k)}.
	$
	Moreover for every $n$ and $k$, it is easy to see that
	\begin{equation}\label{tree1}
	\begin{array}{c}
	y_k^n\leq y_{k+1}^n,
	\quad y^{n+1}_{k}\leq y^n_k\leq y^{n+1}_{k+1},\smallskip\\
	\displaystyle
	y^n_k\leq y  ^n_{k-1}+\sigma^2h+2\sigma\sqrt{y^n_{k-1}h},\quad
	y^{n+1}_k\leq y^n_k+\frac{\sigma^2} 4\,h-\sigma\sqrt{y^n_k h}.
	\end{array}
	\end{equation}
	
	\medskip
	
	\noindent
	\textit{Proof of Proposition \ref{propjumps}.}
	1. The statement is an immediate consequence of the following facts:
	\begin{align}
	&\mbox{if $k_u(n,k)\geq k+2$, then $y^n_k < \theta_*h$,}\label{ass1}\\
	&\mbox{if $k_d(n,k)\leq k-1$, then $y^n_k > \theta^*/h$,}\label{ass2}
	\end{align}
	which we now prove.
	
	First of all, note that $y^n_k+\mu_Y(y^n_k)h= \kappa\theta h +y^n_k(1-\kappa h)$, so by choosing $\bar h = 1/ \kappa$, 	one has 	$y^n_k+\mu_Y(y^n_k)h >0$. Moreover, as   a direct consequence of \eqref{ku}--\eqref{kd} and of \eqref{tree1}, we have that, if $\mu_Y(y^n_k)>0$, then $k_d(n,k)=k$, and if $\mu_Y(y^n_k)<0$, then $k_u(n,k)=k+1$.
	
	Concerning \eqref{ass1}, we obviously assume $y^n_k>0$, so that  $y^{n+1}_{k+1}>0.$ Note that, from \eqref{ku},
	\begin{align*}
	y^n_k+\mu_Y(y^n_k)h > y^{n+1}_{k_u(n,k)-1}\geq y^{n+1}_{k+1}=y^n_k+\frac{\sigma^2}{4}h +\sigma \sqrt{y^n_k h }.
	\end{align*}
	Since $\mu_Y(y^n_k) \leq \kappa\theta$,  we get
	$
	\kappa\theta h > \frac{\sigma^2}{4}h +\sigma \sqrt{y^n_k h }>\sigma \sqrt{y^n_k h},
	$
	from which
	$
	y^n_k < \Big( \frac{\kappa\theta}{\sigma}\Big)^2	h=\theta_*h
	$, and \eqref{ass1} holds.
	
	We prove now \eqref{ass2}.  First of all observe that, if $y^n_k \leq \theta$, then $\mu_Y(y^n_k) >0$ and so $k_d(n,k)=k$. Then we have $y^n_k >\theta$ and from \eqref{vnk} we can assume   $y^{n+1}_k>0$ up to take  $h< ( 2 \sqrt{\theta}/\sigma)^2$.
	Now, by \eqref{kd} we get
	\begin{align*}
	y^n_k+\mu_Y(y^n_k)h < y^{n+1}_{k_d(n,k)+1}\leq  y^{n+1}_{k}=y^n_k+\frac{\sigma^2}{4}h -\sigma \sqrt{y^n_k h },
	\end{align*}
	so that
	$$
	\kappa(\theta-y^n_k)h 
	< \frac{\sigma^2}{4}h -\sigma \sqrt{y^n_k h }.
	$$
	This gives $\kappa y^{n}_k h> \sigma\sqrt{v^{n}_k h} -\frac{\sigma^{2}}{4}\, h +\kappa\theta h$ and, for $h$ small enough, $y^n_k h>\frac{\sigma^2}{4\kappa^2}$, that is, \eqref{ass2} holds.
	
	\smallskip

	2. If $y^n_k \leq \theta_*h$, \eqref{ass2} gives $k_d(n,k)= k$. As regards the up jump, the case $y^{n+1}_{k_u(n,k)}= 0$ is trivial so we consider $y^{n+1}_{k_u(n,k)}>0$.
	In order to prove \eqref{stimasottosoglia}, we consider two possible cases: $k_u(n,k)=k+1$
	and $k_u(n,k)\geq k+2$. In the first case, we have
	\begin{align*}
	y^{n+1}_{k_u(n,k)}-y^n_k= \frac{\sigma^2}{4}h +\sigma \sqrt{y^n_k h }\leq \Big(   \frac{\sigma^2}{4}+\sigma \sqrt{\theta_*}  \Big)h\leq C_\ast h,
	\end{align*}
	and the statement holds. If instead $k_u(n,k)\geq k+2$, then by \eqref{ku} we have
	$$
	y^{n+1}_{k_u(n,k)-1} - y^n_k <\mu_Y(y^n_k)h.
	$$
	We apply the third inequality in \eqref{tree1} (with $n$ replaced by $n+1$ and $k=k_u(n,k)$) and we get
	\begin{align*}
	&0\leq y^{n+1}_{k_u(n,k)}-y^n_k\leq y^{n+1}_{k_u(n,k)-1}+2\sigma \sqrt{y^{n+1}_{k_u(n,k)-1}h } +\sigma^2h-y^n_k \\
	& \quad \leq \mu_Y(y^n_k)h+2\sigma \sqrt{(y^n_k +\mu_Y(y^n_k)h)h}+\sigma^2h
	\leq (\kappa\theta + 2\sigma\sqrt{ \theta_*+\kappa\theta}+\sigma^2)h\leq C_\ast h.
	\end{align*}
	
	3. The statement follows from \eqref{ass1}.
	
	\smallskip
	
	4. Formula \eqref{proba} is proved once we show that the sets $K_u(n,k)=\{k^*\,:\, k+1\leq k^*\leq n+1 \mbox{ and }y^n_k+\mu_Y(y^n_k)h \le y^{n+1}_{ k^*}\}$ and $K_d(n,k)=\{k^*\,:\, 0\leq k^*\leq k \mbox{ and }y^n_k+\mu_Y(y^n_k)h \ge y^{n+1}_{ k^*}\}$ are nonempty. 
	Indeed, 
	if 	$y^n_k> \theta_*h$ then $k_u=k+1$, so $K_u(n,k)\neq \emptyset$.	And if $y^n_k<\theta_*h$,
	\begin{align*}
	y^{n+1}_{n+1}-y^n_k-\mu_Y(y^n_k)h \geq Y_0 -\theta_*h-\kappa\theta h =Y_0 -(\theta_*+\kappa\theta )h>0
	\end{align*}
	for $h< Y_0/(\theta_*+\kappa\theta)$, which gives $k_u(n,k)<n+1$. Therefore   $K_u(n,k)\neq \emptyset$ for every $(n,k)$.
	As regards $K_d(n,k)$, if $y^n_k< \theta^*/h$ then $k_d(n,k)=k$ by Proposition \ref{propjumps}, so that $K_d(n,k)\neq \emptyset$. If instead $y^n_k\geq \theta^*/h$, then
	$$
	y^{n+1}_0-y^n_k-\mu_Y(y^n_k)h\leq Y_0-\frac {\theta^*} h -\kappa\theta h +\kappa y^n_kh\leq Y_0-\frac {\theta^*} h  +\kappa y^n_kh .
	$$
	Recalling that $h=\frac T N$, we note that there exists $C>0$ such that
	\begin{align*}
	y^n_k h\leq y^N_Nh= \Big(\sqrt {Y_0}+\frac{\sigma} 2N\sqrt{h}\Big)^2h= \Big(\sqrt {Y_0}\sqrt{\frac T N }+\frac{\sigma} 2T  \Big)^2\leq C.
	\end{align*}
	Therefore
	$$
	y^{n+1}_0-y^n_k-\mu_Y(y^n_k)h\leq Y_0-\frac {\theta^*} h+\kappa C <0
	$$
	for $h<\frac{\theta^*  }{Y_0+\kappa C}$. So, $K_d(n,k)\neq \emptyset$. Now, by \eqref{ku} and \eqref{kd}, since $K_u(n,k)\neq \emptyset$ and $K_d(n,k)\neq \emptyset$,
$$
	0\leq \frac{\mu_Y(y^n_k)h+ y^n_k-y^{n+1}_{k_d(n,k)} }{y^{n+1}_{k_u(n,k)}-y^{n+1}_{k_d(n,k)}}= 1+ \frac{\mu_Y(y^n_k)h+ y^n_k-y^{n+1}_{k_u(n,k)} }{y^{n+1}_{k_u(n,k)}-y^{n+1}_{k_d(n,k)}}\leq 1.
$$
	\begin{flushright}
		$\square$
	\end{flushright}

\section{Proof of \eqref{lemma-poisson}}\label{SU-Poisson}
For $x\in\R$, let $\lfloor x\rfloor=\sup\{k\in \Z\,:\,k\leq x\}$ denote the integer part. For $N\in\N$, straightforward computations give
$$
\sum_{|n|\leq N}\varphi(n)=\frac 12 (\varphi(N)+\varphi(-N))+\int_{-N}^{N}\varphi(x)dx+\int_{-N}^{N}\Big(x-\lfloor x\rfloor-\frac 12\Big)\varphi'(x)dx.
$$	
We recall that $\varphi(\pm N)\to 0$ as $N\to\infty$ (because $\varphi,\varphi'\in L^1(\R,dx)$). Moreover, 
the Fourier series representation gives
$$
x-\lfloor x\rfloor-\frac 12=\sum_{n\in\Z,n\neq 0}\frac {e^{-2\pi\ii n x}}{2\pi\ii n},\quad x\in\R.
$$
So,
\begin{align*}
\sum_{n\in\Z}\varphi(n)
&=\int_\R \varphi(x)dx+\int_\R\sum_{n\in\Z,n\neq 0}\frac {e^{-2\pi\ii n x}}{2\pi\ii n}\varphi'(x)dx.
\end{align*}
With $\mathfrak{F}[\cdot]$ denoting the Fourier transform, we have $\int_\R e^{-2\pi\ii n x}\varphi'(x)dx=\mathfrak{F}[\varphi'](2\pi n)=2\pi \ii n\mathfrak{F}[\varphi]$ $(2\pi n)$ and $|\mathfrak{F}[\varphi'](2\pi n)|\leq |\frac{\mathfrak{F}[\varphi''](2\pi n)}{2\pi n}|\leq \frac M n$ because $\varphi''\in L^1(\R,dx)$. Thus, we can put the sum outside the integral and 
the statement holds.


\addcontentsline{toc}{section}{References}
	
	\small
	\begin{thebibliography}{}
		
		
		
		\bibitem{ads}
		\textsc{E. Aky\i ld\i r\i m, Y. Dolinsky, H.M.  Soner},   Approximating stochastic volatility by recombinant trees, \textit{Ann. Appl. Probab.}, \textbf{24} (2014), pp. 2176--2205.	
		
		\bibitem{A-MC}
		\textsc{A. Alfonsi}, On the discretization schemes for the CIR (and Bessel squared)
		processes, \textit{Monte Carlo Methods Appl.}, \textbf{11} (2005), pp. 355--467.
		
		\bibitem{A}
		\textsc{A. Alfonsi}, High order discretization schemes for the CIR process: Application to affine term structure and Heston models, \textit{ Math. Comp.}, \textbf{79}  (2010), pp. 209--237. 
		
		\bibitem{AN}
		\textsc{M. Altmayer, A. Neuenkirch},  Discretising the Heston model: an analysis of the weak convergence rate, \textit{IMA J. Numer. Anal.}, \textbf{37}  (2017), pp. 1930--1960. 
		
		\bibitem{and}
		\textsc{L. Andersen}, Simple and efficient simulation of the Heston stochastic volatility model, \textit{J. Comput. Finance}, \textbf{11} (2008), pp. 1--42.
		
		\bibitem{acz} \textsc{E. Appolloni, L. Caramellino, A. Zanette},  A robust tree method
		for pricing American options with CIR stochastic interest rate, {\it IMA J. Manag. Math.}, \textbf{26} (2015), pp. 345--375.
		

		\bibitem{BMO}	
		\textsc{K. Bahlali, B. Mezerdi, Y. Ouknine},  Pathwise uniqueness and approximation of solutions of stochastic differential equations. in: S\'eminaire de Probabilit\'es XXXII,  Lecture Notes in Math., vol. 1686, Springer, Berlin 1998, pp. 166--187.
		


\bibitem{br}
\textsc{V. Bally, C. Rey}, Approximation of Markov semigroups in total variation distance, \textit{Electron. J. Probab.},
\textbf{21}  (2016), no. 12, 44 pp.

		\bibitem{bt} 
\textsc{V. Bally, D. Talay},  The law of the Euler scheme for stochastic differential equations. I. Convergence rate of the
distribution function, \textit{Probab. Theory Related Fields}, \textbf{104}  (1996), pp. 43--60.

		\bibitem{bates}  
		\textsc{D.S. Bates}, Jumps and stochastic volatility: exchange
		rate processes implicit in Deutsch mark options,  \emph{Rev. Fin.}, \textbf{9} (1996), pp. 69--107
		
		


		\bibitem{bossy}
\textsc{M. Bossy, H. Olivero}, 
Strong convergence of the symmetrized Milstein scheme for some CEV-like SDEs,
\textit{Bernoulli}, \textbf{24} (2018), pp. 1995--2042.	
	
		\bibitem{bcz}
		\textsc{M. Briani, L. Caramellino, A. Zanette},   A hybrid approach for the implementation of the Heston model, \textit{IMA J. Manag. Math.}, \textbf{28}  (2017), pp. 467--500.
		
		\bibitem{bcz-hhw}
		\textsc{M. Briani,  L. Caramellino, A. Zanette}, 
		A hybrid tree/finite-difference approach for Heston-Hull-White type models.
		\textit{J. Comput. Finance}, \textbf{21} (2017), pp. 1--45.
		
		\bibitem{bctz}
		\textsc{M. Briani,  L. Caramellino, G. Terenzi, A. Zanette}, 
		On a hybrid method using trees and finite-difference for pricing options in complex models,
		Preprint \texttt{ArXiv:1603.07225},  2017.
		
		\bibitem{bm}
		\textsc{D. Brigo, F. Mercurio}, \textit{Interest Rate Models: Theory and Practice}, Springer 2001.
		
				\bibitem{ckmz}  
		\textsc{C. Chiarella, B. Kang, G. Meyer, A. Ziogas},  The evaluation of American option prices under stochastic
		volatility and jump-diffusion dynamics using the method of lines, \textit{Int. J. Theor.  Appl. Finan.},  \textbf{12} (2009), pp. 393--425.
		
		\bibitem{cgmz}
		\textsc{M. Costabile, M. Gaudenzi, I. Massab\`o, A. Zanette}, Evaluating fair premiums of equity-linked
		policies with surrender option in a bivariate model, \textit{Insurance Math. Econom.},  \textbf{45}  (2009), pp. 286--295.
		
		\bibitem{cir}
		\textsc{J.C. Cox, J.E. Ingersoll, S.A. Ross}, A Theory of the Term Structure of Interest Rates, \textit{Econometrica}, \textbf{53} (1985), pp. 385--407. 
		
		
		\bibitem{dps}
		\textsc{D. Duffie, J. Pan, K. Singleton}, Transform analysis and asset pricing for affine jump-diffusions, \textit{Econometrica}, \textbf{68} (2000), pp. 1343--1376.
		
		\bibitem{ET}
		\textsc{E. Ekstrom, J. Tysk}, The Black and Scholes equation in stochastic volatility models.
		\textit{J. Math. Anal. Appl.}, 368(2) (2010), 498--507.
		
		
		\bibitem{GM}
		\textsc{M.G. Garroni, J.L. Menaldi},  \textit{Green Functions for Second Order Parabolic Integro-Differential Problems}, Pitman Research Notes in Mathematics Series, vol. 275, 1993. 
		
		
		\bibitem{heston}
		\textsc{S.L. Heston},  A Closed-Form Solution for Options with Stochastic Volatility with Applications to Bond and Currency Options. \textit{The Review of Financial Studies}, \textbf{6} (1993), pp. 327--343.
		
		\bibitem{HST}
		\textsc{J.E. Hilliard, A.L. Schwartz, A.L. Tucker}, Bivariate binomial pricing with generalized interest
		rate processes, \textit{J. Financ. Res.}, \textbf{XIX-4}  (1996), pp. 585--602.
		
		\bibitem{it}
		\textsc{A. Itkin}, Efficient Solution of Backward Jump-Diffusion
		PIDEs with Splitting and Matrix Exponentials,  \textit{J. Comput. Finance},
		\textbf{19} (2016), pp. 29--70.
		
		\bibitem{kou}
		\textsc{S.G. Kou}, A Jump-Diffusion Model for Option Pricing, \textit{Management Science}, \textbf{48}  (2002), pp. 1086--1101.
		
		
		\bibitem{lt}
		\textsc{D. Lamberton, G. Terenzi}, Variational formulation of American option prices in the Heston model. \texttt{ArXiv:1711.11311}, 2017
		
		\bibitem{Lax}
		\textsc{P.D. Lax, R.D. Richtmyer},
		Survey of the stability of linear finite difference equations, \emph{Commun. Pure Appl. Math.}, \textbf{9}  (1956),  267--293.

		\bibitem{mer} 
		\textsc{R.C. Merton}, Option pricing when underlying stock returns
		are discontinuous, \textit{J. Financial Econom.}, \textbf{3}  (1976), pp. 125--144.
		
		\bibitem{MP} 
		\textsc{R. Mikulevicius, H. Pragarauskas}, On Cauchy-Dirichlet problem in half-space for linear integro-differential equations in weighted H\"older spaces, \textit{Electron. J. Probab.}, \textbf{10}  (2004), pp. 1398--1416.
		
		
		
		\bibitem{nr} \textsc{D.B. Nelson, K.  Ramaswamy}, Simple binomial processes as diffusion approximations in financial  models. \textit{The Review of Financial Studies},  \textbf{3}  (1990), pp. 393--430.
		
		\bibitem{nv}
		\textsc{H. Nieuwenhuis, M. Vellekoop}, A tree-based method to price American Options in the Heston Model, \textit{J.  Comput. Finance}, \textbf{13}  (2009), 1--21.
		
		 \bibitem{PP}
		 \textsc{G. Pag\`es, J. Printems},  Functional quantization for numerics with an application to option pricing, \textit{Monte Carlo Methods Appl.},  \textbf{11}  (2005), 407--446. 
		 
		 \bibitem{SV}
		\textsc{D.W. Stroock, S.R.S. Varadhan}, Multidimensional Diffusion Processes,  Springer, Berlin 1979.
		
		
		\bibitem{Tian}
		\textsc{Y. Tian},  A reexamination of lattice procedures for interest rate-contingent claims, \textit{ Adv. Futures Options
			Res.},  \textbf{7}  (1994), pp. 87--110.
		
		\bibitem{t} 
		\textsc{J. Toivanen},  A Componentwise Splitting Method for Pricing American Options
		Under the Bates Model, {\it Applied and Numerical Partial Differential Equations}, \textbf{15}  (2010), pp. 213--227.
		
		\bibitem{W}
		\textsc{J. Wei}, 
		Valuing American equity options with a stochastic interest rate: a note. \textit{J. Financ. Eng.}, \textbf{2}  (1996), pp. 195--206.

		\bibitem{z}
		\textsc{C. Zheng}, Weak convergence rate of a time-discrete scheme for the Heston stochastic volatility model, \textit{SIAM J. Numer. Anal.}, \textbf{55}  (2017), pp. 1243--1263.
		
	\end{thebibliography}
\end{document}